\documentclass[11pt,a4paper]{article}

\usepackage[hidelinks,hypertexnames=false,hyperindex=true,pdfpagelabels,linktoc=all]{hyperref}
\newcommand{\email}[1]{\href{mailto:#1}{\texttt{#1}}}

\usepackage{pdfsync}

\usepackage{latexsym,amsfonts,amsmath,amsthm,graphics,multirow}
\usepackage[normalem]{ulem}
\usepackage{enumitem}
\usepackage[usenames,dvipsnames,svgnames,table]{xcolor}

\usepackage{listings}
\definecolor{customblue}{RGB}{235,241,245}
\lstdefinestyle{mystyle}{%
  backgroundcolor=\color{customblue},
  frame=single,
  framerule=0pt,
  basicstyle=\ttfamily\small,
  breaklines=true,
  columns=fullflexible,
  language=c++,
}
\newtheorem{theorem}{Theorem}
\newtheorem{lemma}[theorem]{Lemma}

\newtheorem{remark}[theorem]{Remark}
\newtheorem{example}[theorem]{Example}
\newcommand{\satop}[2]{\stackrel{\scriptstyle{#1}}{\scriptstyle{#2}}}

\newcommand{\bsa}{{\boldsymbol{a}}}

\newcommand{\bsk}{{\boldsymbol{k}}}

\newcommand{\bsb}{{\boldsymbol{b}}}
\newcommand{\bsf}{{\boldsymbol{f}}}
\newcommand{\bsF}{{\boldsymbol{F}}}
\newcommand{\bsx}{{\boldsymbol{x}}}
\newcommand{\bsh}{{\boldsymbol{h}}}

\newcommand{\bsg}{{\boldsymbol{g}}}
\newcommand{\bst}{{\boldsymbol{t}}}
\newcommand{\bsz}{{\boldsymbol{z}}}

\newcommand{\bsone}{{\boldsymbol{1}}}
\newcommand{\bseps}{{\boldsymbol{\varepsilon}}}
\newcommand{\eps}{{\varepsilon}}

\newcommand{\bssigma}{\boldsymbol{\sigma}}

\newcommand{\bszero}{\boldsymbol{0}}
\newcommand{\rd}{\mathrm{d}}
\newcommand{\NN}{\mathbb{N}}
\newcommand{\ZZ}{\mathbb{Z}}

\newcommand{\RR}{\mathbb{R}}

\newcommand{\CC}{\mathbb{C}}

\newcommand{\calA}{{\mathcal{A}}}

\newcommand{\calO}{{\mathcal{O}}}

\newcommand{\calF}{{\mathcal{F}}}

\newcommand{\calM}{{\mathcal{M}}}
\newcommand{\calS}{{\mathcal{S}}}
\newcommand{\bbZ}{{\mathbb{Z}}}
\newcommand{\bbN}{{\mathbb{N}}}
\newcommand{\bbR}{{\mathbb{R}}}

\newcommand{\coefpol}{b}

\newcommand{\size}[1]{\#{#1}}
\newcommand{\sizep}[1]{\#(#1)} 

\newcommand{\cheb}{\eta}
\newcommand{\cosb}{\phi}
\newcommand{\tent}{\varphi_{\rm tent}}

\newcommand{\nozero}{\!\setminus\!\{\bszero\}}

\DeclareMathOperator{\Span}{span} 

\newcommand{\rmd}{{\mathrm{d}}}
\newcommand{\rmi}{{\mathrm{i}}}
\newcommand{\rmu}{{\mathrm{u}}}
\newcommand{\rmv}{{\mathrm{v}}}

\makeatletter
\newcommand*\bigcdot{\mathpalette\bigcdot@{.5}}
\newcommand*\bigcdot@[2]{\mathbin{\vcenter{\hbox{\scalebox{#2}{$\m@th#1\bullet$}}}}}
\makeatother

\definecolor{darkred}{RGB}{139,0,0}
\definecolor{darkgreen}{RGB}{30,130,80}
\definecolor{darkmagenta}{RGB}{139,0,139}
\definecolor{darkorange}{RGB}{180,60,0}

\begin{document}

\title{Function integration, reconstruction and approximation\\ using rank-$1$ lattices}
\author{Frances Y. Kuo \and Giovanni Migliorati \and Fabio Nobile \and Dirk Nuyens%
\footnote{Addresses:
    Frances Y.\ Kuo (\email{f.kuo@unsw.edu.au}), UNSW Sydney, Australia;
    Giovanni Migliorati (\email{migliorati@ljll.math.upmc.fr}), Sorbonne University, France;
    Fabio Nobile (\email{fabio.nobile@epfl.ch}), EPFL, Switzerland;
    Dirk Nuyens (\email{dirk.nuyens@cs.kuleuven.be}, corresponding author), KU Leuven, Belgium.
}}

\date{}

\maketitle
\begin{abstract}
We consider rank-$1$ lattices for integration and reconstruction of
functions with series expansion supported on a finite index set. We
explore the connection between the periodic Fourier space and the
non-periodic cosine space and Chebyshev space, via tent transform and then
cosine transform, to transfer known results from the periodic setting into
new insights for the non-periodic settings. Fast discrete cosine transform
can be applied for the reconstruction phase. To reduce the size of the
auxiliary index set in the associated component-by-component (CBC)
construction for the lattice generating vectors, we work with a
bi-orthonormal set of basis functions, leading to three methods for
function reconstruction in the non-periodic settings. We provide new
theory and efficient algorithmic strategies for the CBC construction. We
also interpret our results in the context of general function
approximation and discrete least-squares approximation.

\bigskip

\smallskip

\end{abstract}

\section{Introduction}

In this paper we consider function integration, reconstruction and
approximation in the periodic and non-periodic settings using rank-$1$
lattices. We explore the connection between three function space settings
to transfer known results on rank-$1$ lattices from the periodic setting
to the non-periodic settings. We obtain necessary and sufficient
conditions on rank-$1$ lattices to achieve the exactness properties we
require in each setting, and we develop efficient algorithms to construct
the generating vectors for rank-$1$ lattices that satisfy these
conditions.

More precisely, we consider functions with absolutely convergent series
expansions with respect to an orthonormal basis, written in the generic
form
\begin{equation} \label{eq:f-gen}
  f \,=\, \sum_{\bsk} \widehat{f}_\bsk \,\alpha_\bsk.
\end{equation}
A large part of this paper is devoted to functions which are fully
supported on a finite index set~$\Lambda$, i.e.,
\begin{equation} \label{eq:fin-gen}
  f \,=\, \sum_{\bsk\in\Lambda} \widehat{f}_\bsk \,\alpha_\bsk.
\end{equation}
We develop methods based on rank-$1$ lattices to exactly integrate such
functions \eqref{eq:fin-gen}, and to exactly reconstruct all series
coefficients $\widehat{f}_\bsk$ in \eqref{eq:fin-gen}. We also consider
the approximation problem for functions \eqref{eq:f-gen} which are not
finitely supported on $\Lambda$.

The three function space settings we consider are as follows:
\begin{itemize}
\item The \textbf{Fourier space} 
contains all
    absolutely convergent Fourier series in the unit cube $[0,1]^d$,
    with exponential basis functions
    $e_\bsh(\bsx)=e^{2\pi\rmi\,\bsh\cdot\bsx}$ and indices
    $\bsh\in\bbZ^d$.

\item The \textbf{cosine space} 
contains all
    absolutely convergent cosine series in $[0,1]^d$, with half-period
    cosine basis functions $\cosb_\bsk$ (see \eqref{eq:cosb} below)
    and nonnegative indices $\bsk\in\bbN_0^d$.

\item The \textbf{Chebyshev space} 
consists of all
    absolutely convergent Chebyshev series in the larger domain
    $[-1,1]^d$, under the Chebyshev measure, with Chebyshev basis
    functions $\cheb_\bsk$ (see \eqref{eq:cheb} below) and also
    nonnegative indices $\bsk\in\bbN_0^d$.
\end{itemize}
To avoid excessive notation we keep to generic notation for the three
spaces wherever possible, including the same `hat' notation for series
coefficients. However, to effectively describe the connection between
spaces, we often distinguish the basis functions $e_\bsh$, $\cosb_\bsk$,
$\cheb_\bsk$, and we often use $\bsh$ and $\bsk$ to contrast indices
containing integers $\bbZ$ or only nonnegative integers $\bbN_0$.

The Fourier space contains periodic functions while the cosine and
Chebyshev spaces contain nonperiodic functions. The Fourier space is
often referred to as the Wiener algebra; it is the standard setting for
analyzing periodic functions, see, e.g., \cite{KDV15,KKP12,KV19,PV16}.
The cosine space is connected to the Fourier space by the \textbf{tent
transform} which is defined by $\tent(x) := 1 - |2x-1|$ for $x\in [0,1]$
and is applied componentwise in $d$ dimensions, see, e.g.,
\cite{CKNS16,DNP14,SNC16}. We show that the composition
$\cosb_\bsk\circ\tent$ is the average over all of those exponential basis
functions $e_\bsh$ for which $(|h_1|,\ldots,|h_d|) = \bsk$ (see
\eqref{eq:tent-mapped-cosb} below). Consequently, \emph{the
tent-transformed cosine space is a subspace of the Fourier space}. Thus we
can apply results from the Fourier space to the cosine space via tent
transform.

The Chebyshev space is related to the cosine space by the \textbf{cosine
transform}, given by $\bsx = \cos(\pi\bsx') \in [-1,1]^d$ for $\bsx'\in
[0,1]^d$, where the cosine function is applied componentwise, and we have
$\cheb_\bsk(\bsx) = \cheb_\bsk(\cos(\pi\bsx')) = \cosb_\bsk(\bsx')$. Thus
\emph{the cosine transform provides an isomorphism between the Chebyshev
space and the cosine space}. Trivially all results from the cosine space
can be carried over to the Chebyshev space.

Rank-$1$ lattices have been well studied for integration, reconstruction
and approximation in the Fourier space; see, e.g., \cite{CKN10,SJ94}
for integration, \cite{Kam13,Kam14} for reconstruction, and
\cite{BKUV17,KDV15,KKP12,KV19,KSW06,KSW08,KWW09c,LH03,ZLH06,ZKH09} for
approximation. Given the \emph{generating vector} $\bsz\in \bbZ^d$, the
$n$ points of a rank-$1$ lattice are specified by
\[
  \bst_i \,=\, \frac{i\bsz \bmod n}{n} \,\in\, [0,1]^d
  \qquad\mbox{for}\quad i=0, \ldots,n-1.
\]
For a Fourier space function $f$, the average of function values at the
lattice points
\[
  Q_n(f) \,:=\, \frac{1}{n} \sum_{i=0}^{n-1} f(\bst_i)
\]
is known as a rank-$1$ lattice rule which is an equal-weight cubature rule
for approximating the integral
\[
  I(f) \,:=\, \int_{[0,1]^d} f(\bsx)\,\rd\bsx.
\]

Rank-$1$ lattices have an important property known  as the ``character
property'' (see \eqref{eq:character-property} below) which states that the
cubature sum of the exponential basis functions $Q_n(e_\bsh)$ can only
take the value of $1$ or $0$, depending on whether or not the dot product
$\bsh\cdot\bsz$ is a multiple of~$n$. Since the integral of the basis
function $I(e_\bsh)$ is $1$ if $\bsh = \bszero$ and is $0$ otherwise, we
easily deduce that a rank-$1$ lattice rule can exactly integrate a
function $f = \sum_{\bsh\in\Lambda} \widehat{f}_\bsh\, e_\bsh$ whose
Fourier series is supported on a finite set $\Lambda\subset\bbZ^d$ if and
only if $\bsh\cdot\bsz$ is not a multiple of $n$ for all nonzero vectors
$\bsh\in\Lambda$. This condition in turn leads to an efficient algorithm
to construct a generating vector $\bsz$ with the exactness property in a
component-by-component fashion. This result is stated later in
Lemma~\ref{lem:fou1}, see also \cite{CKN10}, and it can be said to be the
starting point of all results in this paper. Indeed, the result extends to
function reconstruction on $\Lambda$ where we evaluate all the Fourier
coefficients $\widehat{f}_\bsh$ for $\bsh\in\Lambda$ by a rank-$1$ lattice
rule, and the evaluations can be done using the fast Fourier transform.
Using the character property one can deduce a necessary and
sufficient condition when these Fourier coefficients can be recovered
exactly, thus leading to a constructive algorithm to find suitable
generating vectors by working with the ``difference set''
$\Lambda\ominus\Lambda$ which is obtained by forming all differences of
indices in $\Lambda$. We state this
result later in Lemma~\ref{lem:fou2}, which was first proved with varying generality in
\cite{Kam13,Kam14,PV16}. The idea has been further extended to the
construction of ``multiple rank-$1$ lattices'' in \cite{KV19}, where the
benefits of multiple reconstruction lattices are combined strategically to
achieve the same goal with a reduced overall number of sampling nodes; we
do not go down this path.

The connection between the Fourier space and the cosine space allows us to
apply the theory of rank-$1$ lattices to the cosine space by tent
transform. We can obtain necessary and sufficient conditions for
\emph{tent-transformed rank-$1$ lattices} to achieve the integral
exactness and function reconstruction properties in the cosine space, see
Lemmas~\ref{lem:cos1} and~\ref{lem:cos2} below (see also \cite{SNC16} for
part of Lemma~\ref{lem:cos2}). In the case of function reconstruction on a
finite index set $\Lambda\subset\bbN_0^d$, we end up having to work with
quite a large auxiliary index set $\calM(\Lambda)\oplus\calM(\Lambda)$ in
our component-by-component construction of the lattice generating vector,
where $\calM(\Lambda)$ denotes the ``mirrored set'' obtained from
$\Lambda$ by including all sign changes of indices in $\Lambda$, while the
$\oplus$ then means that we form the sum of all indices from
$\calM(\Lambda)$; we call this \textbf{plan A}. To improve on the
computational efficiency of plan~A, we show by working with a
bi-orthonormal set of basis functions that we can achieve function
reconstruction on $\Lambda$ with a weaker condition which means working
with a smaller auxiliary index set $\Lambda\oplus\calM(\Lambda)$, see
Lemma~\ref{lem:plan-b} below; we call this \textbf{plan B}. We also relax
the algorithm to not necessarily recover the normalization of the basis
functions to arrive at \textbf{plan C}, which achieves the same
reconstruction property at a lower computational cost, see
Lemma~\ref{lem:plan-c} below. All three plans for function reconstruction
in the cosine space can be computed using the fast discrete cosine
transform.

The isomorphism between the cosine space and the Chebyshev space allows us
to take all results from the cosine space to the Chebyshev space,
including plans A, B, C. We arrive at \emph{tent-transformed and then
cosine-transformed rank-$1$ lattices}, which in the case of $n$ being even
is also known as ``Chebyshev lattices'', see, e.g., \cite{CP11,PV15},
although we do not adopt this terminology. Our plan C for the Chebyshev
space with even $n$ is essentially the approach in \cite{PV15}; in this
paper we do not require $n$ to be even.

\subsection{Layout of the paper and highlight of new results}

In Section~\ref{sec:fou} we review results on rank-$1$ lattices for
integration and function reconstruction on a finite index set in the
Fourier space, referencing essential results from
\cite{CKN10,Kam13,Kam14,PV16}.

In Section~\ref{sec:cos} we introduce the cosine space and consider
integration and function reconstruction, with three plans for achieving
exact function reconstruction using rank-$1$ lattices with varying costs.
Except for the if-part of Lemma~\ref{lem:cos2} and
Lemma~\ref{lem:unique-points} which was proved in \cite{SNC16}, all
remaining results in this section are new, including plan B and plan C for
function reconstruction and the applicability of fast discrete cosine
transform.

In Section~\ref{sec:che} we present the corresponding results for the
Chebyshev space. Lemma~\ref{lem:che-plan-c} for plan C with even $n$ turns
out to be equivalent to the approach in \cite{PV15}. However, the precise
connection to the cosine space via cosine transform and in turn the
precise connection to the Fourier space via tent transform are both new
interpretations here, and they lead to broader implications in the
Chebyshev space. In particular, the multiplicity of the transformed points
under these interpretations are known explicitly for $n$ both even and
odd, and fast discrete cosine transform can be applied for all $n$.

Section~\ref{sec:cbc} is devoted to the theory and algorithmic aspect of
the component-by-component (CBC) construction for lattice generating
vectors achieving various conditions needed for the exactness properties.
As the theoretical justification for the CBC construction,
Theorem~\ref{thm:cbc} generalises previous results proved in
\cite{CKN10,Kam13,Kam14} and provides a cheaper variant of the algorithm
when building up the index set, while Theorem~\ref{thm:cbc-plan-c} is new
and specific to plan C. The systematic way to combine two different
approaches (namely, the ``brute force'' approach and the ``elimination''
approach, to be explained in Section~\ref{sec:cbc}) in a mixed CBC
construction is new. Strategies for storage and a ``smart lookup'' to
efficiently search through difference and/or sum involving mirrored sets
are also new.

Finally in Section~\ref{sec:app} we interpret our results in the context
of approximation of general functions that are not necessarily supported
on a finite index set, and compare them with discrete least-squares
approximation as analysed in \cite{CCMNT15,CDL13,MNST14,MN15,NXZ14}. We
mention other known results in function approximation based on rank-$1$
lattices (see, e.g., \cite{KWW08,KWW09a,WW99,WW01} for general results and
\cite{BKUV17,CKNS16,KDV15,KV19,KSW06,KSW08,KWW09c,LH03,ZLH06,ZKH09}
for rank-$1$ lattices).

We end the introduction with setting the notation on multiindices and
introducing some special index sets.

\subsection{Notation on multiindices and special index sets}

Throughout this paper we use $\size{}$ to denote the cardinality of a set.
For any multiindex $\bsk\in\bbZ^d$, we write $|\bsk|_0 := \size{\{1\le
j\le\ d: k_j\ne 0\}}$ for the number of nonzero indices in $\bsk$. For
$\bsk,\bsk'\in\bbZ^d$, $\bsk' \le \bsk$ is to be interpreted
componentwise, i.e., $k_j' \le k_j$ for all $j$.

For $\bssigma \in \{\pm 1\}^d$ and $\bsk\in\bbZ^d$, we write
$\bssigma(\bsk) := (\sigma_1 k_1,\ldots,\sigma_d k_d)$ to mean that we
apply the sign changes in $\bssigma$ componentwise to $\bsk$. For any
$\bsk\in\bbZ^d$ we use
\[
  \calS_\bsk \,:=\,
  \left\{ \bssigma \in \{\pm 1\}^d : \sigma_j = +1 \mbox{ for each $j$ for which $k_j = 0$} \right\}
\]
to denote a set of \emph{unique sign changes} for $\bsk$. Then clearly we
have $\size{\calS_\bsk} = 2^{|\bsk|_0}$.

We will consider index sets with some special properties:
\begin{itemize}
\item An index set $\Lambda\subset\bbN_0^d$ is \emph{downward closed}
    if $\bsk' \in \Lambda$ whenever $\bsk' \le \bsk$ and $\bsk \in
    \Lambda$. This means that from every $\bsk \in \Lambda$ we can
    move towards $\bszero$ along the coordinate axes without finding a
    $\bsk' \not\in \Lambda$. Analogous definition holds with
    $\bbN_0^d$ replaced by $\bbZ^d$.

\item An index set $\Lambda\subset\bbZ^d$ is \emph{centrally
    symmetric} if $-\bsk \in \Lambda$ whenever $\bsk \in \Lambda$.

\item An index set $\Lambda\subset\bbZ^d$ is \emph{fully sign
    symmetric} if $\bssigma(\bsk) \in \Lambda$ whenever $\bsk \in
    \Lambda$ and $\bssigma\in\{\pm 1\}^d$.

\item An index set $\Lambda\subset\bbN_0^d$ is an (\emph{anisotropic})
    \emph{tensor product} set if there exist $\bsa,\bsb\in\bbN_0^d$
    such that $\Lambda = \{\bsk\in\bbN_0^d : \bsa\le\bsk\le\bsb\}$.
    Analogous definition holds with $\bbN_0^d$ replaced by $\bbZ^d$.
\end{itemize}

For any index set $\Lambda \subset \bbN_0^d$ or $\Lambda \subset \bbZ^d$,
we denote its largest component in magnitude by
\[
  \max(\Lambda) \,:=\, \max_{\bsk\in\Lambda} \max_{1\le j\le d} |k_j|,
\]
and we define
\begin{align*}
  \Lambda \oplus \Lambda &\,:=\, \{\bsk+\bsk' : \bsk,\bsk'\in\Lambda\}
  && \mbox{(``sum set'')}, \\
  \Lambda \ominus \Lambda &\,:=\, \{\bsk-\bsk' : \bsk,\bsk'\in\Lambda\}
  && \mbox{(``difference set'')}, \\
  \calM(\Lambda) &\,:=\, \left\{\bssigma(\bsk) : \bsk\in\Lambda, \bssigma\in\{\pm 1\}^d\right\}
  \,=\, \bigcup_{\bsk\in\Lambda} \left\{\bssigma(\bsk) : \bssigma\in\calS_\bsk\right\}
  && \mbox{(``mirrored set'')}.
\end{align*}
If $\Lambda\subset\bbZ^d$ is centrally symmetric, then $\Lambda \oplus
\Lambda = \Lambda \ominus \Lambda$. If $\Lambda\in\bbZ^d$ is fully sign
symmetric, then $\calM(\Lambda) = \Lambda$. For $\Lambda \subset \ZZ^d$,
$\Lambda \ominus \Lambda$ is always centrally symmetric (since both $\bsk
- \bsk'$ and $\bsk' - \bsk$ belong to $\Lambda \ominus \Lambda$ when
$\bsk, \bsk' \in \Lambda$).

Trivially we have
\[
 \sizep{\Lambda \oplus \Lambda} \,\le\, (\size{\Lambda})^2, \qquad
 \sizep{\Lambda \ominus \Lambda} \,\le\, (\size{\Lambda})^2,
 \qquad\mbox{and}\qquad
 \size{\calM(\Lambda)} \,\le\, \sum_{\bsk\in\Lambda} 2^{|\bsk|_0}
 \,\le\, 2^d\,\size{\Lambda}.
\]
The squaring effect in the upper bounds for the sum/difference set cannot
be avoided in general, since even if all multiindices in $\Lambda$ fall on
the first two axes (i.e., all components of $\bsk$ are zero except for one
of $k_1$ and $k_2$), the sum/difference sets will contain a large
rectangle (so there is a lower bound of the same order). On the other
hand, the $2^d$ factor in the upper bound for $\calM(\Lambda)$ can
sometimes be improved, as shown in the forthcoming
Lemma~\ref{lem:downward} and Example~\ref{example} below. We also need the
next Lemma~\ref{comb_bound_downward_Z_general}, whose proof uses induction
arguments from \cite{Mig15,CCMNT15}.

\begin{lemma}
\label{comb_bound_downward_Z_general} In any dimension $d$, given any
$\Lambda \subset \mathbb{Z}^d$ downward closed and any polynomial
$p(n)=\sum_{k=0}^{\eta} \coefpol_k n^k$ of degree $\eta\geq 0$ with
nonnegative coefficients $\coefpol_0\leq 1$ and all $\coefpol_k \leq
\binom{\eta +1}{k}$, it holds
\begin{equation}
\sum_{\bsk \in \Lambda} \prod_{j=1}^{d} p( \vert k_j\vert) \leq (\# \Lambda)^{\eta +1}.
\label{eq:thesis_lemma_weakly_down_closed}
\end{equation}
Moreover, it holds
\begin{equation}
  \sum_{\bsk\in\Lambda} 2^{|\bsk|_0} \,\le\, (\size{\Lambda})^{\ln 3/\ln 2}.
\label{eq:thesis_lemma_weakly_down_closed_cheb}
\end{equation}
\end{lemma}
\begin{proof}
When $\eta=0$ the result holds true. Consider then the case $\eta \geq 1$.
Every downward closed set $\Lambda \subset \mathbb{Z}^d$ can be seen as a
set $\widetilde{\Lambda} \subset \mathbb{N}_0^{2d}$ constructed in the
following way. Start with $\widetilde{\Lambda}=\emptyset$. For any $\bsk
\in \Lambda$, define the sets $C_\bsk:=\{ 1\le j\le d : k_j<0 \}\subseteq
\{1,\ldots,d\}$ and $U_\bsk:=\{1,\ldots,d\}\setminus C_\bsk$. Then define
$\widetilde{\bsk} \in \mathbb{N}_0^{2d}$ by setting $\widetilde{k}_j= k_j$
and $\widetilde{k}_{j+d}=0$ for all $j\in U_\bsk$, and $\widetilde{k}_j=0$
and $\widetilde{k}_{j+d}=-k_j$ for all $j\in C_\bsk$. Finally add
$\widetilde{\bsk}$ to $\widetilde{\Lambda}$. Notice that this algorithm
establishes a one-to-one correspondence between the elements of $\Lambda$
and $\widetilde{\Lambda}$, and therefore $\#\Lambda=\#
\widetilde{\Lambda}$. By construction $\widetilde{\Lambda}$ is also
downward closed in $\mathbb{N}_0^{2d}$. Applying Theorem 1 from
\cite{Mig15} to the set $\widetilde{\Lambda} \in \mathbb{N}_0^{2d}$ we
obtain \eqref{eq:thesis_lemma_weakly_down_closed}.

For the proof of \eqref{eq:thesis_lemma_weakly_down_closed_cheb}, as
above, starting from $\Lambda$ we construct the downward closed set
$\widetilde{\Lambda}\subset \mathbb{N}_0^{2d}$  such that
$\#\widetilde{\Lambda}=\#\Lambda$, and then apply Lemma 3.3 from
\cite{CCMNT15} to the set $\widetilde{\Lambda}$.
\end{proof}

\begin{lemma} \label{lem:downward}
If $\Lambda\subset\bbZ^d$ is downward closed then
\[
  \max_{\bsk\in\Lambda} 2^{|\bsk|_0} \,\le\, \size{\Lambda},
  \quad
  \sum_{\bsk\in\Lambda} 2^{|\bsk|_0} \,\le\, (\size{\Lambda})^{\ln 3/\ln 2},
  \quad\mbox{and}\quad
  \size{\calM(\Lambda)} \,\le\, \min\left(2^d\,\size{\Lambda},(\size{\Lambda})^{\ln 3/\ln 2}\right).
\]
\end{lemma}

\begin{proof}
For the first bound, since $\Lambda$ is downward closed, for any $\bsk\in
\Lambda$, the set $\Lambda$ will include the hyper-rectangle with $\bsk$
and the origin as corners. Thus $\size{\Lambda} \ge \prod_{1\le j\le
d,\,k_j\ne 0} (1 + |k_j|) \ge 2^{|\bsk|_0}$. The second bound is proved in
Lemma~\ref{comb_bound_downward_Z_general}. The third bound is an immediate
consequence of the second bound.
\end{proof}

Since $\size{\Lambda}$ most likely grows with $d$, in general it is not
obvious which of $2^d\,\size{\Lambda}$ or $(\size{\Lambda})^{\ln 3/\ln 2}$
is a better bound for $\size{\calM(\Lambda)}$. If $\size{\Lambda}$ can be
bounded independently of $d$, then most likely so can
$\size{\calM(\Lambda)}$.

\begin{example} \label{example}
Consider a ``weighted'' index set of ``degree'' $m\in\bbN$ defined by
(see, e.g., \cite{CKN10}) $\Lambda \,=\, \{\bsk\in \bbN_0^d : r(\bsk) \le
m\}$, where $r(\bsk)$ is given by
\begin{equation} \label{eq:weighted}
  \max_{1\le j\le d} \frac{k_j}{\beta_j},
  \quad
  \sum_{1\le j\le d} \frac{k_j}{\beta_j},
  \quad\mbox{or}\quad
  \prod_{j=1}^d \max\left(1,\frac{k_j}{\beta_j}\right),
\end{equation}
with $1=\beta_1\ge \beta_2\ge \cdots > 0$ and $\sum_{j=1}^\infty \beta_j <
\infty$. The first example is an anisotropic tensor product set and is the
largest of the three examples. We have $\size{\Lambda} = \prod_{j=1}^d (1
+ \lfloor \beta_j m\rfloor) \le \exp(m \sum_{j=1}^\infty \beta_j)$ and
$\size{\calM(\Lambda)} = \prod_{j=1}^d (1 + 2\lfloor \beta_j m\rfloor) \le
\exp(2m \sum_{j=1}^\infty \beta_j)$, so both are bounded independently
of~$d$. Their ratio satisfies
\[
  \frac{\size{\calM(\Lambda)}}{\size{\Lambda}}
  \,=\, \prod_{j=1}^d \left(1 + \frac{\lfloor \beta_j m\rfloor}{1+ \lfloor \beta_j m\rfloor}\right)
  \,\le\, \min\bigg(2^J, \exp\bigg(m\sum_{j=1}^\infty \beta_j\bigg)\bigg),
\]
where $J$ is the ``truncation dimension'' such that $\beta_j m < 1$ for
all $j\ge J$. If we have $\lambda\in (0,1]$ such that $\sum_{j=1}^\infty
\beta_j^\lambda <\infty$, then $\beta_J < J^{-1/\lambda}
(\sum_{j=1}^\infty \beta_j^\lambda)^{1/\lambda}$ and it suffices to take
$J = m^\lambda (\sum_{j=1}^\infty \beta_j^\lambda)$. Both upper bounds on
the ratio grow exponentially with $m$.

The third example in \eqref{eq:weighted} is the smallest of the three. Its
mirror set is commonly referred to as the ``Zaremba cross'' or
``hyperbolic cross'', see, e.g., \cite{CKN10,Kam13}. For all $\tau> 1$ we
have
\begin{align*}
  m+1 \,\le\, \size{\Lambda}
  &\,\le\, m^\tau \prod_{j=1}^d \left(1 + \zeta(\tau)\beta_j^\tau\right)
  \,\le\, m^\tau \exp\bigg(\zeta(\tau)\sum_{j=1}^\infty \beta_j^\tau\bigg), \\
  2m+1 \le \size{\calM(\Lambda)}
  &\,\le\, m^\tau \prod_{j=1}^d \left(1 + 2\zeta(\tau)\beta_j^\tau\right)
  \,\le\, (\size{\Lambda})^\tau \exp\bigg(2\zeta(\tau)\sum_{j=1}^\infty \beta_j^\tau\bigg),
\end{align*}
where $\zeta(\tau) := \sum_{k=1}^\infty k^{-\tau}$ is the Riemann zeta
function. Since $\tau$ can be arbitrarily close to $1$,
$\size{\calM(\Lambda)}$ is essentially of the same order as
$\size{\Lambda}$, both are bounded independently of $d$. The upper bound
on $\size{\calM(\Lambda)}$ is proved in \cite{KSW06}.
\end{example}

\section{Periodic setting based on trigonometric polynomials}
\label{sec:fou}

\subsection{Fourier series}

We start by considering periodic functions on $[0,1]^d$. Let $\calF^{\rm
Four}$ denote the space of complex-valued functions defined on $[0,1]^d$
with absolutely converging Fourier series:
$$
  \calF^{\rm Four}
  :=
  \Bigl\{ f \in L^2 \mathbin{\Big|} f : [0,1]^d \to \mathbb{C},\;
  f(\bsx) = \sum_{\bsh \in \bbZ^d} \widehat{f}_\bsh \, e^{2\pi\rmi\,\bsh\cdot\bsx}
  \text{ and } \sum_{\bsh \in \bbZ^d} |\widehat{f}_\bsh| < \infty \Bigr\},
$$
where $\bsh\cdot\bsx := h_1x_1 + \cdots + h_dx_d$ is the usual dot product
and $\widehat{f}_\bsh$ are the Fourier coefficients. We equip $\calF^{\rm
Four}$ with the usual $L^2$ inner product
\begin{equation} \label{eq:inner1}
  \langle f_1, f_2 \rangle
  \,:=\, \int_{[0,1]^d} f_1(\bsx) \overline{f_2(\bsx)}\,\rd\bsx.
\end{equation}
The exponential functions form an orthonormal basis
\[
  e_\bsh(\bsx) \,:=\, e^{2\pi\rmi\,\bsh\cdot\bsx}
\]
satisfying $\langle e_\bsh,e_{\bsh'}\rangle = \delta_{\bsh,\bsh'}$, where
the Kronecker delta function yields $1$ if $\bsh=\bsh'$ and $0$ if
$\bsh\ne\bsh'$. The Fourier coefficients are given by
\[
 \widehat{f}_\bsh
 \,:=\, \langle f, e_\bsh \rangle
 \,=\, \int_{[0,1]^d} f(\bsx)\, e^{-2\pi\rmi\,\bsh\cdot\bsx}\,\rd\bsx,
 \qquad \bsh \in \bbZ^d.
\]
The norm of $f$ satisfies $\|f\|^2 \,=\, \int_{[0,1]^d} |f(\bsx)|^2
\,\rd\bsx \,=\, \sum_{\bsh\in\bbZ^d} |\widehat{f}_\bsh|^2$.

\subsection{Fourier coefficients by cubature}

For $f\in\calF^{\rm Four}$ we define the integral operator
\[
  I(f) \,:=\, \int_{[0,1]^d} f(\bsx)\,\rd\bsx.
\]
Later we will seek a cubature formula $Q_n(f)$ which uses linear
combinations of $n$ evaluations of $f$ to approximate $I(f)$. We define a
discrete inner product
\[
  \langle f_1,f_2\rangle_n \,:=\, Q_n(f_1\, \overline{f_2})
\]
as an approximation to \eqref{eq:inner1}.

Given an arbitrary finite index set $\Lambda\subset\bbZ^d$, we consider
the subspace $\calF^{\rm Four}_\Lambda$ of all functions whose Fourier
series is supported solely on $\Lambda$, i.e.,
\begin{align} \label{eq:f}
  \text{for } f \in \calF^{\rm Four}_\Lambda: \qquad
  f(\bsx) \,=\, \sum_{\bsh\in\Lambda} \widehat{f}_\bsh\,e^{2\pi\rmi\,\bsh\cdot\bsx}.
\end{align}
Implicitly, this means that all other Fourier coefficients of $f$ are
zero, i.e., $\widehat{f}_\bsh$ = 0 for $\bsh\notin\Lambda$.

In this paper we will demand one or both of the following related
properties on the cubature formula:

\begin{itemize}
\item \textbf{Integral exactness.} We want our cubature formula to be
    exact for all functions which are supported solely on $\Lambda$,
    i.e., we want $Q_n(f) = I(f)$ for all $f\in\calF^{\rm
    Four}_\Lambda$. This holds if and only if
\[ 
  Q_n(e_\bsh) \,=\, I(e_\bsh)
  \,=\, \delta_{\bsh,\bszero}
  \qquad\mbox{for all } \bsh\in\Lambda,
\] 
i.e., our cubature formula integrates exactly all basis functions
$e^{2\pi\rmi\,\bsh\cdot\bsx}$ with $\bsh\in\Lambda$.

\item \textbf{Function reconstruction.} Instead of \eqref{eq:f} we
    consider
\[ 
  f^a(\bsx) \,=\,
  \sum_{\bsh\in\Lambda} \widehat{f}_\bsh^a \,e^{2\pi\rmi\,\bsh\cdot\bsx},
\] 
where each Fourier coefficient $\widehat{f}_\bsh = \langle
f,e_\bsh\rangle = I(f\,e_{-\bsh})$ in \eqref{eq:f} is replaced by the
cubature formula $\widehat{f}_\bsh^a := \langle f,e_\bsh\rangle_n =
Q_n(f\,e_{-\bsh})$. We demand the ``non-aliasing'' condition that
\[
  \widehat{f}_\bsh^a \,=\, \widehat{f}_\bsh
  \qquad\mbox{for all } \bsh\in\Lambda
  \mbox{ and } f\in\calF^{\rm Four}_\Lambda,
\]
so that $f^a$ is a reconstruction of $f$. (If other coefficients
$\widehat{f}_{\bsh'}$ with $\bsh'\ne\bsh$ contribute to
$\widehat{f}_\bsh^a$ then this is called ``aliasing''.) Using the
linearity of $Q_n$, we then have
\begin{align*}
 \widehat{f}_\bsh^a &\,=\,Q_n(f\,e_{-\bsh}) \\
 &\,=\, Q_n \bigg(\bigg(\sum_{\bsh'\in\Lambda} \widehat{f}_{\bsh'}\,e_{\bsh'}\bigg) e_{-\bsh}\bigg)
 \,=\, \sum_{\bsh'\in\Lambda} \widehat{f}_{\bsh'}\,Q_n(e_{\bsh'-\bsh})
 = \widehat{f}_\bsh
 \quad\mbox{for all } \bsh\in\Lambda
 \mbox{ and } f\in\calF^{\rm Four}_\Lambda.
\end{align*}
This holds if and only if
\[
  Q_n(e_{\bsh'-\bsh})
  \,=\, \langle e_{\bsh'}, e_\bsh\rangle
  \,=\, \delta_{\bsh,\bsh'}
 \qquad\mbox{for all } \bsh,\bsh'\in\Lambda,
\]
which is equivalent to
\[ 
  Q_n(e_{\bsh})
  \,=\, I(e_{\bsh})
  \,=\, \delta_{\bsh,\bszero}
 \qquad\mbox{for all } \bsh\in\Lambda\ominus\Lambda,
\] 
i.e., our cubature formula integrates exactly all basis functions
$e^{2\pi\rmi\,\bsh\cdot\bsx}$ with $\bsh\in\Lambda\ominus\Lambda$.
\end{itemize}

\subsection{Rank-1 lattice rules}

Consider now the cubature formula given by \emph{rank-$1$ lattices}
\begin{equation} \label{eq:lat}
  Q_n(f) \,:=\, \frac{1}{n} \sum_{i=0}^{n-1} f\left(\frac{i\bsz \bmod n}{n}\right),
\end{equation}
where $\bsz\in\bbZ^d$ is an integer vector known as the \emph{generating
vector}. It is easy to verify the ``character property'' that for any
$\bsh\in\bbZ^d$,
\begin{align}\label{eq:character-property}
  Q_n (e_\bsh) \,=\, \frac{1}{n} \sum_{i=0}^{n-1} e^{2\pi\rmi\, i\bsh\cdot\bsz/n}
  \,=\,
  \begin{cases}
  1 & \mbox{if } \bsh\cdot\bsz\equiv_n 0, \\
  0 & \mbox{otherwise},
  \end{cases}
\end{align}
where the notation $a \equiv_n b$ means that $(a\bmod n) = (b\bmod n)$.
This leads to the well-known lattice cubature error formula for
$f\in\calF^{\rm Four}$
\[
  Q_n(f) - I(f)
  \,=\, \sum_{\satop{\bsh\in\bbZ^d\setminus\{\bszero\}}{\bsh\cdot\bsz\equiv_n 0}} \widehat{f}_\bsh.
\]
The set of integer vectors $\{\bsh\in\bbZ^d : \bsh\cdot\bsz \equiv_n 0\}$
is known as the \emph{dual lattice}. Clearly the cubature rule is exact
for a function $f$ solely supported on $\Lambda$ if and only if the dual
lattice does not contain any index from $\Lambda\setminus\{\bszero\}$. We
know how to obtain such a lattice rule generating vector using a
component-by-component construction.

\begin{lemma}[Integral exactness] \label{lem:fou1}
Let $\Lambda \subset \bbZ^d$ be an arbitrary index set. A lattice rule
with $n$ points and generating vector $\bsz$ integrates exactly all
functions $f \in \calF^{\rm Four}_\Lambda$ solely supported on $\Lambda$
if and only if
\[ 
  \bsh\cdot\bsz\not\equiv_n 0
  \qquad\mbox{for all }
  \bsh \in \Lambda \setminus\{\bszero\}.
\] 
Such a generating vector $\bsz$ can be constructed component-by-component
if $n$ is a prime satisfying
$$
  n > \max\left\{ \frac{\sizep{\Lambda\nozero}}{\kappa} + 1 , \,\max(\Lambda) \right\}
  ,
$$
with $\kappa=2$ if $\Lambda$ is centrally symmetric and $\kappa=1$
otherwise.
\end{lemma}

\begin{proof}
The result for some standard anisotropic, downward closed and centrally
symmetric sets $\Lambda$ can be found in Cools, Kuo \& Nuyens
\cite{CKN10}. A proof for general index sets is provided later in
Section~\ref{sec:cbc}, see Theorem~\ref{thm:cbc} and Remark~\ref{rem:cbc}.
A similar proof can be found in K\"ammerer~\cite{Kam14}.
\end{proof}

\begin{lemma}[Function reconstruction] \label{lem:fou2}
Let $\Lambda \subset \bbZ^d$ be an arbitrary index set. A lattice rule
$Q_n$ with $n$ points and generating vector $\bsz$ reconstructs exactly
the Fourier coefficients of all functions $f \in \calF^{\rm Four}_\Lambda$
solely supported on $\Lambda$, by
\[
 \widehat{f}_\bsh \,=\, \widehat{f}_\bsh^a \,:=\,Q_n(f\,e_{-\bsh})
 \qquad\mbox{for all } \bsh\in\Lambda,
\]
if and only if
\[ 
 \bsh\cdot\bsz\not\equiv_n 0 \qquad\mbox{for all } \bsh\in (\Lambda\ominus\Lambda)\nozero,
\] 
which is equivalent to
\[ 
 \bsh\cdot\bsz\not\equiv_n \bsh'\cdot\bsz \qquad\mbox{for all } \bsh,\bsh'\in \Lambda \mbox{ with }\bsh\ne\bsh'.
\] 
Such a generating vector $\bsz$ can be constructed component-by-component
if $n$ is a prime satisfying
\[ 
  n >
  \max\left\{
  \frac{\sizep{\Lambda\ominus\Lambda} + 1}{2}, \, 2 \max(\Lambda)
  \right\}.
\] 
\end{lemma}

\begin{proof}
The result follows directly from Theorem~\ref{thm:cbc}, noting that the
difference set $\Lambda\ominus\Lambda$ is centrally symmetric and contains
$\bszero$, and therefore
$\frac{1}{2}\sizep{(\Lambda\ominus\Lambda)\setminus\{\bszero\}} + 1 =
\frac{1}{2} (\sizep{\Lambda\ominus\Lambda} + 1)$. Alternatively, the
result for $\Lambda$ a hyperbolic cross index set can be found in
K\"ammerer \cite{Kam13}, while the result for any arbitrary index set
$\Lambda$ can be found in K\"ammerer \cite{Kam14} and Potts \& Volkmer
\cite[Theorem~2.1]{PV16}.
\end{proof}

We end this section by the very interesting property that mapping from
function values to Fourier coefficients and the other way around can be
done using a one-dimensional fast Fourier transform.

\begin{lemma}\label{lem:FFT}
  Let $\bsz$ be a generating vector for an $n$-point rank-$1$ lattice satisfying
  the reconstruction property on an arbitrary index set $\Lambda \subset \ZZ^d$
  according to Lemma~\ref{lem:fou2}.
  For a function $f \in \calF^{\rm Four}_\Lambda$ solely supported on $\Lambda$
  we can compute
  \begin{align*}
  \begin{array}{ll}
    \mbox{coefficients from function values:} \\\hline
    \text{// prepare function value vector $\bsf \in \CC^n$} \\
    \text{\textup{\texttt{for}} } i\in \{0,\ldots,n-1\} \text{\textup{\texttt{:}}}\\
    \qquad f_i = f((i\bsz \bmod{n})/n) \\
    \\[3mm]
    \text{// compute coefficient vector $\bsF \in \CC^n$} \\
    \bsF = \mathrm{FFT}(\bsf) \\[3mm]
    \text{// } \widehat{f}_\bsh \text{ is given by } F_{(\bsh\cdot\bsz \bmod{n})} \\
  \end{array}
  \quad\left|\quad
  \begin{array}{ll}
    \mbox{function values from coefficients:} \\\hline
    \text{// prepare coefficient vector $\bsF \in \CC^n$} \\
    \text{$\bsF = \bszero \in \CC^n$} \\
    \text{\textup{\texttt{for}} } \bsh \in \Lambda \text{\textup{\texttt{:}}}\\
    \qquad F_{(\bsh\cdot\bsz \bmod{n})} = \widehat{f}_\bsh \\[3mm]
    \text{// compute function value vector $\bsf \in \CC^n$} \\
    \bsf = \mathrm{IFFT}(\bsF) \\[3mm]
    \text{// } f_i \text{ gives the value of } f((i\bsz \bmod{n})/n) \\
  \end{array}\right.
  \end{align*}
  where $\bsf \in \CC^n$ is a vector containing function values
  and $\bsF \in \CC^n$ is a vector containing Fourier coefficients.
  Here $\mathrm{FFT}$ and $\mathrm{IFFT}$ are the one-dimensional fast Fourier
  transform and its inverse, respectively, with a normalization $1/n$ for $\mathrm{FFT}$
  and $1$ for $\mathrm{IFFT}$; both mappings have cost $\calO(n \log(n))$.
\end{lemma}
\begin{proof}
  This follows from expanding the formula for $\widehat{f}_\bsh^a$
  in Lemma~\ref{lem:fou2}.
  Each $\bsh \in \Lambda$ will correspond to a unique value of $\bsh \cdot \bsz \bmod{n}$ by the
  non-aliasing condition in Lemma~\ref{lem:fou2}.
  The other Fourier coefficients are zero by the assumption that
  $f \in \calF^{\rm Four}_\Lambda$ is solely supported on $\Lambda$.
\end{proof}

\begin{remark}\label{remark:Fourier-other-kappas}
If the function $f$ has wider support in the Fourier space than just
$\Lambda$, then the vector~$\bsF$ resulting from the evaluation $\bsF =
\mathrm{FFT}(\bsf)$ will not necessarily be zero at positions $F_\kappa$
when $\kappa$ does not correspond to a value of $\bsh\cdot\bsz \bmod{n}$
for some $\bsh \in \Lambda$. This is due to the aliasing effect from
$\bsh$ outside of $\Lambda$ and this will also contaminate all other
components of $\bsF$. It is possible to extend the index set to full size
$n$ while still keeping the reconstruction property on the extended index
set such that all values in $\bsF$ can be interpreted as Fourier
coefficients. This technique has been used, e.g., in
\cite{LH03,MKS12,SSN19,SN19}.
\end{remark}

\section{Nonperiodic setting based on half-period cosines} \label{sec:cos}

\subsection{Cosine series}

The cosine basis functions are a complete and orthonormal basis for
$L^2([0,1]^d)$:
\begin{align} \label{eq:cosb}
  \cosb_\bsk(\bsx)
  \,:=\,
  \sqrt2^{|\bsk|_0} \prod_{j=1}^d \cos(\pi k_j x_j)
  ,
  \qquad
  \bsk \in \NN_0^d
  ,
\end{align}
where $|\bsk|_0$ denotes the count of the nonzero entries in the vector
$\bsk$, and we have $\langle \cosb_\bsk,\cosb_{\bsk'}\rangle =
\delta_{\bsk,\bsk'}$. The ``cosine space'' $\calF^{\rm cos}$ consists of
nonperiodic real-valued functions on $[0,1]^d$ with absolutely converging
cosine series:
$$
  \calF^{\rm cos} := \Big\{ f \in L^2 \mathbin{\Big|}  f: [0,1]^d\to\bbR,\;
    f(\bsx) = \sum_{\bsk \in \NN_0^d} \widehat{f}_\bsk \, \cosb_\bsk(\bsx)
    \text{ and }
    \sum_{\bsk \in \NN_0^d} |\widehat{f}_\bsk| < \infty
  \Big\}
$$
where the cosine coefficients are
$$
  \widehat{f}_\bsk
  \,:=\,
  \int_{[0,1]^d} f(\bsx) \, \cosb_\bsk(\bsx) \, \rmd{\bsx}
  .
$$
This space was studied for integration and approximation in
\cite{DNP14,SNC16,CKNS16}. Even though the cosine basis is a complete
orthonormal system for $L^2([0,1]^d)$, it does not allow the
representation of arbitrary polynomials.

\subsection{Cosine coefficients by cubature}

As in Section~\ref{sec:fou}, for a given finite index set
$\Lambda\subset\bbN_0^d$ we consider the subspace $\calF^{\rm
cos}_\Lambda$ of all functions whose cosine series is supported solely on
$\Lambda$, i.e.,
\begin{align} \label{eq:f2}
  \text{for } f \in \calF^{\rm cos}_\Lambda: \qquad
  f(\bsx) \,=\, \sum_{\bsk\in\Lambda} \widehat{f}_\bsk\,\cosb_\bsk(\bsx),
\end{align}
and we are interested in two related properties on the cubature formula:
\begin{itemize}
\item \textbf{Integral exactness.} We want $Q_n(f) = I(f)$ for all
    $f\in\calF^{\rm cos}_\Lambda$, which holds if and only if
\begin{align} \label{eq:exact2}
  Q_n(\cosb_\bsk) \,=\, I(\cosb_\bsk)
  \,=\, \delta_{\bsk,\bszero}
  \qquad\mbox{for all } \bsk\in\Lambda.
\end{align}

\item \textbf{Function reconstruction.} We replace each cosine
    coefficient $\widehat{f}_\bsk = \langle f,\cosb_\bsk\rangle =
    I(f\,\cosb_\bsk)$ in \eqref{eq:f2} by the cubature formula
    $\widehat{f}_\bsk^a := \langle f,\cosb_\bsk\rangle_n =
    Q_n(f\,\cosb_\bsk)$, and demand the non-aliasing condition
\begin{align*}
 \widehat{f}_\bsk^a
 &\,=\, Q_n \bigg(\bigg(\sum_{\bsk'\in\Lambda} \widehat{f}_{\bsk'}\,\cosb_{\bsk'}\bigg) \cosb_{\bsk}\bigg)
 \,=\, \sum_{\bsk'\in\Lambda} \widehat{f}_{\bsk'}\,Q_n(\cosb_{\bsk'}\,\cosb_\bsk)
 = \widehat{f}_\bsk
 \quad\mbox{for all } \bsk\in\Lambda
 \mbox{ and } f\in\calF^{\rm cos}_\Lambda,
\end{align*}
which holds if and only if
\begin{align} \label{eq:recon2}
  Q_n(\cosb_\bsk\,\cosb_{\bsk'})
  \,=\, \langle \cosb_\bsk, \cosb_{\bsk'}\rangle
  \,=\, \delta_{\bsk,\bsk'}
 \qquad\mbox{for all } \bsk,\bsk'\in\Lambda.
\end{align}
\end{itemize}

Unlike the Fourier case where a product of two basis functions is another
basis function, here the condition \eqref{eq:recon2} is not
straightforward to simplify, except when the index set $\Lambda$ is
downward closed. In the next section we will obtain necessary and
sufficient conditions for function reconstruction by connecting with the
Fourier space, without the restriction to downward closed index sets.

\subsection{Connection with the Fourier case via tent transform}

Below we will obtain sufficient conditions to achieve \eqref{eq:exact2}
and \eqref{eq:recon2} in the cosine space by utilizing a known connection
with the Fourier case via the so-called ``tent transform''(see, e.g.,
\cite{Hic02})
$$
  \tent : [0,1] \to [0,1], \qquad \tent(x) \,:=\, 1 - |2 x - 1|
  .
$$
The tent transform is a Lebesgue preserving transformation and therefore a
componentwise mapping of
\[
  \bsx' \,=\, \tent(\bsx) \,:=\, (\tent(x_1),\ldots,\tent(x_d))
\]
yields
$$
  I(f \circ \tent)
  \,=\, \int_{[0,1]^d} f(\tent(\bsx)) \,\rmd{\bsx}
  \,=\, \int_{[0,1]^d} f(\bsx') \,\rmd{\bsx'}
  \,=\, I(f).
$$
To get a sense of how this transformation works, it is informative to
consider the univariate case:
\begin{align*}
  \int_0^1 f(\tent(x))\,\rmd x
  &\,=\, \int_0^{1/2} f(2x)\,\rmd x + \int_{1/2}^1 f(2-2x)\,\rmd x \\
  &\,=\, \int_0^{1} f(x')\,(\tfrac{1}{2}\,\rmd x') + \int_1^0 f(x')\,(-\tfrac{1}{2}\,\rmd x')
  \,=\, \int_0^1 f(x')\,\rmd x'.
\end{align*}

In the following, we recall the definition of the ``mirrored'' index set
associated with the index set $\Lambda$, $\calM(\Lambda) :=
\{\bssigma(\bsk): \bsk\in\Lambda, \bssigma\in\{\pm1\}^d\} =
\bigcup_{\bsk\in\Lambda} \{\bssigma(\bsk) : \bssigma\in \calS_\bsk\}$,
where $\calS_\bsk$ is the set of all unique sign changes of $\bsk$.

\begin{lemma}[Integral exactness -- sufficiency] \label{lem:cos1-pre}
Let $\Lambda \subset \NN_0^d$ be an arbitrary index set. If a cubature
rule $Q_n^*(f) = \sum_{i=0}^{n-1} w_i^*\,f(\bst_i^*)$ integrates exactly
all Fourier basis functions $e_\bsh$ with $\bsh\in\calM(\Lambda)$, then
the cubature rule $Q_n(f) = \sum_{i=0}^{n-1} w_i\,f(\bst_i)$ with $w_i =
w_i^*$ and $\bst_i = \tent(\bst_i^*)$ integrates exactly all cosine space
functions $f \in \calF^{\rm cos}_\Lambda$ solely supported on $\Lambda$.
\end{lemma}

\begin{proof}
For any $\bsk \in \bbN_0^d$ we can write
\[ 
  \cosb_\bsk(\bsx)
  \,=\,
  \sqrt2^{|\bsk|_0} \prod_{j=1}^d \cos(\pi k_j x_j)
  \,=\,
  \frac{1}{\sqrt2^{|\bsk|_0}} \sum_{\bssigma \in \calS_\bsk} \exp(\pi\,\rmi\, \bssigma(\bsk) \cdot \bsx),
\] 
which follows from expanding the product of $\cos(\theta_j) =
(e^{\rmi\theta_j}+e^{-\rmi\theta_j})/2$ for those $\theta_j\ne 0$.
Furthermore, since $\cos(\pi k\, \tent(x)) = \cos(2\pi k x)$ for all $k
\in \bbN_0$, we also have
\begin{align} \label{eq:tent-mapped-cosb}
  &\cosb_\bsk(\tent(\bsx))
  \,=\,
  \frac{1}{\sqrt2^{|\bsk|_0}} \sum_{\bssigma \in \calS_\bsk} e_{\bssigma(\bsk)}(\bsx)
  .
\end{align}
Thus if we have a cubature rule $Q^*_n$ which integrates exactly all
Fourier basis functions $e_{\bssigma(\bsk)}$ for all sign changes of
$\bsk\in \Lambda$, then
\begin{align}
  Q_n(\cosb_\bsk) \,:=\, Q^*_n(\cosb_\bsk\circ \tent)
  &\,=\,  \frac{1}{\sqrt2^{|\bsk|_0}} \sum_{\bssigma \in \calS_\bsk} Q_n^*(e_{\bssigma(\bsk)}) \label{eq:link1} \\
  &\,=\,  \frac{1}{\sqrt2^{|\bsk|_0}} \sum_{\bssigma \in \calS_\bsk} \delta_{\bssigma(\bsk),\bszero}
  \,=\,  \delta_{\bsk,\bszero}, \nonumber
\end{align}
as required for integral exactness in \eqref{eq:exact2}. The cubature rule
$Q_n$ is obtained from $Q_n^*$ by applying the tent-transform to the points.
\end{proof}

\begin{lemma}[Function reconstruction -- sufficiency]\label{lem:cos2-pre}
Let $\Lambda \subset \NN_0^d$ be an arbitrary index set. If a cubature
rule $Q_n^*(f) = \sum_{i=0}^{n-1} w_i^*\,f(\bst_i^*)$ integrates exactly
all Fourier basis functions $e_\bsh$ with $\bsh\in \calM(\Lambda) \oplus
\calM(\Lambda)$, then the cubature rule $Q_n(f) = \sum_{i=0}^{n-1}
w_i\,f(\bst_i)$ with $w_i = w_i^*$ and $\bst_i = \tent(\bst_i^*)$
reconstructs exactly the cosine coefficients of all cosine space functions
$f \in \calF^{\rm cos}_\Lambda$ solely supported on $\Lambda$.
\end{lemma}

\begin{proof}
For any $\bsk,\bsk'\in\bbN_0^d$ we have from \eqref{eq:tent-mapped-cosb}
that
\begin{align*}
  \cosb_\bsk(\tent(\bsx))\,\cosb_{\bsk'}(\tent(\bsx))
  &\,=\, \frac{1}{\sqrt{2}^{|\bsk|_0+|\bsk'|_0}} \sum_{\bssigma\in\calS_\bsk} \sum_{\bssigma'\in\calS_{\bsk'}}
  e_{\bssigma(\bsk)+\bssigma'(\bsk')}(\bsx).
\end{align*}
Thus if we have a cubature rule $Q^*_n$ which integrates exactly all
Fourier basis functions $e_{\bssigma(\bsk)+\bssigma'(\bsk')}$ for all sign
changes of $\bsk,\bsk'\in \Lambda$, then
\begin{align} \label{eq:link2}
  Q_n(\cosb_\bsk\,\cosb_{\bsk'}) \,:=\, Q^*_n((\cosb_\bsk\,\cosb_{\bsk'})\circ \tent)
  &\,=\, \frac{1}{\sqrt{2}^{|\bsk|_0+|\bsk'|_0}} \sum_{\bssigma\in\calS_\bsk} \sum_{\bssigma'\in\calS_{\bsk'}}
  Q_n^*(e_{\bssigma(\bsk)+\bssigma'(\bsk')}) \\
  &\,=\, \frac{1}{\sqrt{2}^{|\bsk|_0+|\bsk'|_0}} \sum_{\bssigma\in\calS_\bsk} \sum_{\bssigma'\in\calS_{\bsk'}}
  \delta_{\bssigma(\bsk)+\bssigma'(\bsk'),\bszero} \nonumber\\
  &\,=\, \frac{1}{\sqrt{2}^{|\bsk|_0+|\bsk'|_0}} \sum_{\bssigma\in\calS_\bsk}
  \sum_{\satop{\bssigma'\in\calS_{\bsk'}}{\sigma_j' = -\sigma_j \text{ when } k_j\ne 0}} \delta_{\bsk,\bsk'}
  \qquad \,=\, \delta_{\bsk,\bsk'}, \nonumber
\end{align}
which is the reconstruction property~\eqref{eq:recon2}. In the penultimate
step we used the property that $\bssigma(\bsk)+\bssigma'(\bsk') = \bszero$
if and only if $\bsk = \bsk'$ and $\sigma_j' = -\sigma_j$ whenever $k_j\ne
0$.
\end{proof}

Now we consider the situation where the cubature rule $Q_n^*$ in
Lemma~\ref{lem:cos1-pre} and Lemma~\ref{lem:cos2-pre} is a rank-1 lattice
rule \eqref{eq:lat}. In this case, the corresponding cubature rule $Q_n(f)
= Q_n^*(f \circ \tent)$ is often called a \emph{tent-transformed lattice
rule}, given explicitly by
\[ 
  Q_n (f) \,:=\, \frac{1}{n} \sum_{i=0}^{n-1} f\left(\tent\left(\frac{i\bsz \bmod n}{n}\right)\right).
\] 
The character property \eqref{eq:character-property} of lattice rules
enables us to conclude that the implications in Lemma~\ref{lem:cos1-pre}
and Lemma~\ref{lem:cos2-pre} also hold in the opposite direction, and we
obtain necessary and sufficient conditions for tent-transformed lattice
rules to achieve our desired properties. Lemma~\ref{lem:cos1} and the
``only if'' part of Lemma~\ref{lem:cos2} have not been explicitly stated
in the literature.

\begin{lemma}[Integral exactness] \label{lem:cos1}
Let $\Lambda \subset \NN_0^d$ be an arbitrary index set. A
tent-transformed lattice rule $Q_n$ of a lattice rule $Q_n^*$ with $n$
points and generating vector $\bsz$ integrates exactly all cosine space
functions $f \in \calF^{\rm cos}_\Lambda$ solely supported on $\Lambda$ if
and only if
\[
  \bsh\cdot\bsz\not\equiv_n 0
  \qquad\mbox{for all }
  \bsh \in \calM(\Lambda)\nozero.
\]
Such a generating vector $\bsz$ can be constructed component-by-component
if $n$ is a prime satisfying
$$
  n > \max\left\{ \frac{\sizep{\calM(\Lambda)\nozero}}{2} + 1 , \,
  \max(\Lambda) \right\}
  .
$$
\end{lemma}

\begin{proof}
The ``if'' direction follows by combining Lemma~\ref{lem:fou1} with
Lemma~\ref{lem:cos1-pre}. To prove the ``only if'' direction, we observe
from the character property \eqref{eq:character-property} that the terms
$Q_n^*(e_{\bssigma(\bsk)})$ on the right-hand side of \eqref{eq:link1} can
only take the values of $1$ or $0$ so there can be no cancelation. In
particular, when $\bsk\ne\bszero$, if $Q_n(\cosb_\bsk)$ is $0$ on the
left-hand side of \eqref{eq:link1} then necessarily all terms
$Q_n^*(e_{\bssigma(\bsk)})$ are $0$ on the right-hand side of
\eqref{eq:link1}, which implies $\bssigma(\bsk)\cdot\bsz\not\equiv_n 0$.
When $\bsk=\bszero$ both sides of \eqref{eq:link1} are equal to $1$, and
trivially $Q_n(\cosb_{\bszero})=1$ implies $Q_n^*(e_{\bszero})=1$. The CBC
result follows from Theorem~\ref{thm:cbc}, noting that $\calM(\Lambda)$ is
centrally symmetric.
\end{proof}

\begin{lemma}[Function reconstruction -- plan A] \label{lem:cos2}
Let $\Lambda \subset \NN_0^d$ be an arbitrary index set. A
tent-transformed lattice rule $Q_n$ of a lattice rule $Q_n^*$ with $n$
points and generating vector $\bsz$ reconstructs exactly the cosine
coefficients of all cosine space functions $f \in \calF^{\rm cos}_\Lambda$
solely supported on $\Lambda$, by
\[
  \widehat{f}_\bsk \,=\, \widehat{f}^a_\bsk \,:=\, Q_n(f\,\phi_\bsk) \,=\, Q_n^*((f\,\phi_\bsk)\circ\tent)
  \qquad\mbox{for all } \bsk\in\Lambda,
\]
if and only if
\[
  \bsh\cdot\bsz\not\equiv_n 0
  \qquad\mbox{for all }
  \bsh \in \calM(\Lambda) \oplus \calM(\Lambda)\nozero.
\]
Such a generating vector $\bsz$ can be constructed component-by-component
if $n$ is a prime satisfying
\[ 
  n > \max\left\{ \frac{\sizep{\calM(\Lambda)\oplus \calM(\Lambda)}+1}{2}, \,
  2 \max(\Lambda) \right\}
  .
\] 
\end{lemma}

\begin{proof}
As in the previous proof, the ``if'' direction follows by combining
Lemma~\ref{lem:fou2} with Lemma~\ref{lem:cos2-pre}. When $\bsk\ne\bsk'$,
if $Q_n(\cosb_\bsk\, \cosb_{\bsk'})=0$ on the left-hand side of
\eqref{eq:link2}, then necessarily all terms
$Q_n^*(e_{\bssigma(\bsk)+\bssigma'(\bsk')})$ are $0$ on the right-hand
side of \eqref{eq:link2}, since the only permissible values are $0$ or~$1$
due to the character property \eqref{eq:character-property}. When
$\bsk=\bsk'$, if $Q_n(\cosb_\bsk\, \cosb_\bsk)=1$ on the left-hand side
of~\eqref{eq:link2}, then we have
\begin{align*} 
  1 \,=\, Q_n(\cosb_\bsk\,\cosb_\bsk)
  &\,=\, \frac{1}{2^{|\bsk|_0}}
  \sum_{\satop{\bssigma,\bssigma'\in\calS_\bsk}{\bssigma(\bsk)+\bssigma'(\bsk)=\bszero}}
  Q_n^*(1)
  + \frac{1}{2^{|\bsk|_0}}
  \sum_{\satop{\bssigma,\bssigma'\in\calS_\bsk}{\bssigma(\bsk)+\bssigma'(\bsk)\ne\bszero}}
  Q_n^*(e_{\bssigma(\bsk)+\bssigma'(\bsk)})  \nonumber\\
  &\,=\, 1  + \frac{1}{2^{|\bsk|_0}}
  \sum_{\satop{\bssigma,\bssigma'\in\calS_\bsk}{\bssigma(\bsk)+\bssigma'(\bsk)\ne\bszero}}
  Q_n^*(e_{\bssigma(\bsk)+\bssigma'(\bsk)}).
\end{align*}
Necessarily, all terms $Q_n^*(e_{\bssigma(\bsk)+\bssigma'(\bsk)})$ must be
zero for $\bssigma(\bsk)+\bssigma'(\bsk)\ne\bszero$. Hence we conclude
that $(\bssigma(\bsk)+\bssigma'(\bsk))\cdot\bsz \not\equiv_n 0$ for all
$\bsk,\bsk',\bssigma,\bssigma'$ satisfying
$\bssigma(\bsk)+\bssigma'(\bsk')\ne\bszero$. The CBC result follows from
Theorem~\ref{thm:cbc}.
\end{proof}

The tent transformation ``stretches and folds the domain'' so that
essentially one half of the lattice points will land on top of the other
half. This is given precisely by the property that
\[
  \tent(\bst_i) \,=\, \tent(\bst_{n-i})
  \qquad\mbox{for } 1\le i < n/2.
\]
There is one point at the origin $\bst_0$ which will not be duplicated.
When $n$ is even, there is one other point $\tent(\bst_{n/2})$ which will
not be duplicated.
All other points have a multiplicity of two provided that the generating
vector includes at least one component $z_j$ such that $\gcd(n,z_j)=1$.
Typically in a CBC construction we set $z_1 = 1$. This is sufficient to
ensure uniqueness.

\begin{lemma}\label{lem:unique-points}
A rank-$1$ lattice rule with $n$ points and generating vector $\bsz \in
\ZZ_n^d$, where $\gcd(n,z_j)=1$ for some $j$, has $\lfloor n / 2 +
1\rfloor$ unique points after tent transform.
\end{lemma}

\begin{proof}
See Suryanarayana, Nuyens \& Cools \cite{SNC16}.
\end{proof}

\subsection{Alternative approach for function reconstruction}\label{sec:bi-orthonormal}

In this subsection we use an alternative approach for function
reconstruction. There are two essential ingredients, which we will
separate into \textbf{plan B} and \textbf{plan C} below. Firstly, we make
use of the tent transform and bi-orthonormality to switch to a simpler set
of functions. Secondly, we allow ``self-aliasing'' of the cubature rule to
relax bi-normality and correct for this normalization afterward. We shall
see in the next section that this alternative approach has a connection
with the method in Potts \& Volkmer \cite{PV15}.

\begin{remark}[Orthonormal and bi-orthonormal families]
For function reconstruction we demand that the inner product of all basis
functions in our support set are exactly represented by replacing the
integral by a cubature rule. An alternative is to be able to exactly
represent the inner product of all basis functions with another set of
orthogonal functions which have the bi-orthonormal property. In general,
if $\{\rmu_\bsk\}$ is an orthonormal basis with $\langle
\rmu_{\bsk},\rmu_{\bsk'} \rangle = \delta_{\bsk,\bsk'}$, and
$\{\rmv_\bsk\}$ is an orthogonal set with the bi-orthonormal property
$\langle \rmu_\bsk,\rmv_{\bsk'} \rangle = \delta_{\bsk,\bsk'}$, while
$\langle \rmv_{\bsk}, \rmv_{\bsk'} \rangle = d_\bsk\, \delta_{\bsk,\bsk'}$
with $d_\bsk$ not necessarily equal to $1$, then the coefficients of a
function $f \in \Span\{ \rmu_\bsk \}$ can be calculated by the inner
product against either $\{\rmu_\bsk\}$ or $\{\rmv_\bsk\}$ since $\langle
f, \rmu_\bsk \rangle = \langle f, \rmv_\bsk \rangle$. A cubature rule
which can exactly calculate the inner products $\langle \rmu_\bsk,
\rmv_{\bsk'} \rangle = \delta_{\bsk,\bsk'}$ for all $\bsk,\bsk'$ in our
support set then also has the reconstruction property. We will make use of
such a bi-orthonormal property below.
\end{remark}

\begin{lemma} \label{lem:dot-innerproduct}
For any $f \in \calF^{\rm cos}$ and $\bsk \in \NN_0^d$, we can write the
cosine coefficients in multiple ways
\begin{align*} 
 \widehat{f}_\bsk
 &\,=\,
 \langle f , \cosb_\bsk \rangle \nonumber \\
 &\,=\,
 \langle f \circ \tent , \cosb_\bsk \circ \tent \rangle
 \,=\,
 \langle f \circ \tent , \sqrt{2}^{|\bsk|_0} e_\bsk \rangle
 \,=\,
 \langle f \circ \tent , \sqrt{2}^{|\bsk|_0} \cos(2\pi\bsk\cdot \bigcdot) \rangle \\
 &\,=\,
 \langle f \circ \tent , \sqrt{2}^{|\bsk|_0} e_{\bssigma(\bsk)} \rangle
 \,=\,
 \langle f \circ \tent , \sqrt{2}^{|\bsk|_0} \cos(2\pi\bssigma(\bsk)\cdot\bigcdot) \rangle
 \quad\mbox{for all } \bssigma\in \calS_\bsk \nonumber
 .
\end{align*}
\end{lemma}

\begin{proof}
Using the Lebesgue preserving property of the tent transform and
\eqref{eq:tent-mapped-cosb}, we can write
\begin{align*}
  \widehat{f}_\bsk
  \,=\,
  \int_{[0,1]^d} f(\tent(\bsx)) \, \cosb_{\bsk}(\tent(\bsx)) \,\rmd{\bsx}
  &\,=\,
  \frac{1}{\sqrt{2}^{|\bsk|_0}} \sum_{\bssigma\in\calS_\bsk}
  \int_{[0,1]^d} f(\tent(\bsx)) \, e_{\bssigma(\bsk)}(\bsx) \,\rmd{\bsx}.
\end{align*}
For each integral, we apply the change of variables $x_j' = x_j$ if
$\sigma_j = 1$ and $x_j' = 1-x_j$ if $\sigma_j = -1$, and use the
properties $\exp(2\pi\rmi\, \bssigma(\bsk)\cdot\bsx) = \exp(2\pi\rmi\,
\bsk\cdot\bssigma(\bsx)) = \exp(2\pi\rmi\, \bsk\cdot \bsx')$ and
$\tent(\bsx) = \tent(\bsx')$ to deduce that
\[
  \int_{[0,1]^d} f(\tent(\bsx)) \, e_{\bssigma(\bsk)}(\bsx) \,\rmd{\bsx}
  \,=\, \int_{[0,1]^d} f(\tent(\bsx')) \, e_\bsk(\bsx') \,\rmd{\bsx'}
  \qquad\mbox{for all } \bssigma\in \calS_\bsk.
\]
Thus all integrals are equal regardless of the sign changes on $\bsk$.
Furthermore, since $f$ is a real-valued function, all its cosine
coefficients will be real. Hence we may replace the exponential function
$e_{\bssigma(\bsk)}$ by its real part
$\cos(2\pi\bssigma(\bsk)\cdot\bigcdot)$.
\end{proof}

By considering the special case of $f = \cosb_{\bsk'}$ in
Lemma~\ref{lem:dot-innerproduct} we obtain
\begin{align} \label{eq:2inner}
 \delta_{\bsk,\bsk'} \,=\,
 \langle \cosb_{\bsk'} , \cosb_{\bsk} \rangle
 \,=\,
 \langle \cosb_{\bsk'} \circ \tent , \sqrt{2}^{|\bsk|_0} \cos(2\pi\bsk\cdot\bigcdot) \rangle
 \qquad\mbox{for all }
 \bsk,\bsk'\in\bbN_0^d
 ,
\end{align}
i.e., the functions $\{\rmu_\bsk = \phi_\bsk \circ \tent\}$ and
$\{\rmv_\bsk = \sqrt{2}^{|\bsk|_0} \cos(2\pi\bsk\cdot\bigcdot)\}$ are
bi-orthonormal in $L^2$. Instead of demanding a cubature rule with
exactness for the first inner product in \eqref{eq:2inner} (as we did in
\eqref{eq:recon2}), below we seek a cubature rule with exactness for the
second inner product in \eqref{eq:2inner} (see \eqref{eq:claim} below),
thus preserving bi-orthonormality.

\begin{lemma}[Function reconstruction -- plan B] \label{lem:plan-b}
Let $\Lambda \subset \NN_0^d$ be an arbitrary index set. A lattice rule
$Q_n^*$ with $n$ points and generating vector $\bsz$ reconstructs exactly
the cosine coefficients of all cosine space functions $f\in\calF^{\rm
cos}_\Lambda$ solely supported on $\Lambda$, by
\begin{align}
 \widehat{f}_\bsk \,=\,
 \widehat{f}_\bsk^b \label{eq:plan-b}
  &\,:=\,
  Q_n^*((f\circ \tent) \, \sqrt{2}^{|\bsk|_0} \cos(2\pi \bsk \cdot \bigcdot))
  \qquad
  \text{for all } \bsk \in \Lambda,
\end{align}
if and only if
\begin{equation} \label{eq:plan-b-iff}
 \bssigma(\bsk')\cdot\bsz \not\equiv_n \bsk\cdot\bsz
 \qquad
 \mbox{for all } \bsk,\bsk' \in \Lambda, \, \bssigma\in\calS_{\bsk'},\, \bssigma(\bsk')\ne\bsk,
\end{equation}
which is equivalent to
\begin{equation} \label{eq:plan-b-iff-2}
 \bsh\cdot\bsz \not\equiv_n 0
 \qquad
 \mbox{for all } \bsh \in (\Lambda\oplus\calM(\Lambda))\nozero.
\end{equation}
Such a generating vector $\bsz$ can be constructed component-by-component
if $n$ is a prime satisfying
\[ 
  n >
  \max\left\{ \sizep{\Lambda \oplus \calM(\Lambda)}, \,
  2 \max(\Lambda) \right\}.
\] 
\end{lemma}

\begin{proof}
Substituting the cosine series of $f$ into~\eqref{eq:plan-b}, it follows
that we have exact reconstruction of the cosine coefficients, i.e.,
\begin{align*}
  \widehat{f}_\bsk^b
  \,=\, \sum_{\bsk' \in \Lambda} \widehat{f}_{\bsk'}\,
    Q_n^*((\cosb_{\bsk'}\circ\tent)\, \sqrt{2}^{|\bsk|_0} \cos(2\pi\bsk\cdot\bigcdot))
  \,=\, \widehat{f}_\bsk
  \quad
  \text{for all } f\in\calF^{\rm cos}_\Lambda \mbox{ and } \bsk \in \Lambda,
\end{align*}
if and only if
\begin{align} \label{eq:claim}
  Q_n^*((\cosb_{\bsk'}\circ\tent) \, \sqrt{2}^{|\bsk|_0} \cos(2\pi \bsk \cdot \bigcdot))
  \,=\,
  \delta_{\bsk',\bsk}
  \qquad
  \text{for all } \bsk, \bsk' \in \Lambda.
\end{align}
It remains to prove that \eqref{eq:claim} holds if and only if
\eqref{eq:plan-b-iff} holds.

Using \eqref{eq:tent-mapped-cosb} and $\cos(2\pi\bsk\cdot\bigcdot) =
(e_\bsk + e_{-\bsk})/2$, we can write
\begin{align} \label{eq:link-b}
  Q_n^*((\cosb_{\bsk'}\circ\tent) \, \sqrt{2}^{|\bsk|_0} \cos(2\pi \bsk \cdot \bigcdot))
  &\,=\, \frac{\sqrt{2}^{|\bsk|_0}}{2\cdot \sqrt{2}^{|\bsk'|_0}} \sum_{\bssigma\in\calS_{\bsk'}}
  \left(Q_n^*(e_{\bssigma(\bsk')+\bsk}) + Q_n^*(e_{\bssigma(\bsk')-\bsk})\right) \nonumber\\
  &\,=\, \frac{\sqrt{2}^{|\bsk|_0}}{2\cdot \sqrt{2}^{|\bsk'|_0}} \sum_{\bssigma\in\calS_{\bsk'}}
  \left(Q_n^*(e_{-\bssigma(\bsk')+\bsk}) + Q_n^*(e_{\bssigma(\bsk')-\bsk})\right),
\end{align}
where it is valid to replace one $\bssigma(\bsk')$ by $-\bssigma(\bsk')$
since we sum over all unique sign changes.

Consider first the case $\bsk \ne\bsk'$. Then $\bssigma(\bsk')\ne \bsk$.
By the character property \eqref{eq:character-property} we know that
$Q_n^*(e_{\pm(\bssigma(\bsk')-\bsk)})$ can only take the values of $1$ or
$0$. Thus \eqref{eq:link-b} is equal to~$0$ if and only if all terms
$Q_n^*(e_{\pm(\bssigma(\bsk')-\bsk)})$ are~$0$, which holds following the
character property if and only if $ \bssigma(\bsk')\cdot\bsz \not\equiv_n
\bsk\cdot\bsz$.

Consider now the case $\bsk = \bsk'$. Then we can rewrite
\eqref{eq:link-b} as
\begin{align} \label{eq:link-b2}
  Q_n^*((\cosb_{\bsk}\circ\tent) \, \sqrt{2}^{|\bsk|_0} \cos(2\pi \bsk \cdot \bigcdot))
  &\,=\, 1
  + \frac{1}{2} \sum_{\satop{\bssigma\in\calS_\bsk}{\bssigma(\bsk)\ne\bsk}}
  \Big(Q_n^*(e_{-\bssigma(\bsk)+\bsk}) + Q_n^*(e_{\bssigma(\bsk)-\bsk})\Big).
\end{align}
Using the character property as before, we conclude that
\eqref{eq:link-b2} is equal to~$1$ if and only if all terms
$Q_n^*(e_{\pm(\bssigma(\bsk)-\bsk)})$ are $0$ whenever
$\bssigma(\bsk)\ne\bsk$, and in turn this means that
$\bssigma(\bsk)\cdot\bsz \not\equiv_n \bsk\cdot\bsz$ except for when
$\bssigma(\bsk)=\bsk$. Combining all conditions, we conclude that
\eqref{eq:claim} holds if and only if \eqref{eq:plan-b-iff} holds.

Finally \eqref{eq:plan-b-iff} is clearly equivalent to
\eqref{eq:plan-b-iff-2}. The condition on $n$ then follows from
Theorem~\ref{thm:cbc}, noting that the set $\Lambda \oplus \calM(\Lambda)$
includes $\bszero$ but is not centrally symmetric.
\end{proof}

In the next lemma we propose another modification which allows
``self-aliasing'' in the lattice rule with respect to sign changes (see
$\bssigma(\bsk)\cdot\bsz \equiv_n \bsk\cdot\bsz$ in \eqref{eq:plan-c-ck}
below). Consequently, the right-hand side of \eqref{eq:claim} for the case
$\bsk = \bsk'$ can be an integer $c_\bsk$, not necessarily $1$ (see
\eqref{eq:claim-ck} below). In other words, the cubature rule no longer
preserves bi-normality, with normalization to be corrected by this factor
$c_\bsk$.

\begin{lemma}[Function reconstruction -- plan C] \label{lem:plan-c}
Let $\Lambda \subset \NN_0^d$ be an arbitrary index set. A lattice rule
$Q_n^*$ with $n$ points and generating vector $\bsz$ reconstructs exactly
the cosine coefficients of all cosine space functions $f\in\calF^{\rm
cos}_\Lambda$ solely supported on $\Lambda$, by
\begin{align}
 \widehat{f}_\bsk \,=\,
 \widehat{f}_\bsk^c \nonumber 
  &\,:=\,
\frac{Q_n^*((f\circ \tent) \, \sqrt{2}^{|\bsk|_0} \cos(2\pi \bsk \cdot \bigcdot))}{c_\bsk}, \quad\mbox{with} \\
  c_\bsk \label{eq:plan-c-ck}
  &\,:=\,
  \size{ \big\{ \bssigma\in\calS_\bsk \;:\;
  \bssigma(\bsk)\cdot\bsz  \equiv_n \bsk \cdot \bsz \big\}}
  \quad
  \text{for all } \bsk \in \Lambda,
\end{align}
if and only if
\begin{equation} \label{eq:plan-c-iff}
  \bssigma(\bsk')\cdot\bsz \not\equiv_n \bsk\cdot\bsz
  \qquad\mbox{for all }
  \bsk,\bsk' \in \Lambda, \, \bssigma\in\calS_{\bsk'},\, \bsk\ne \bsk'.
\end{equation}
Such a generating vector $\bsz$ can be constructed component-by-component
if $n$ is a prime satisfying
\[ 
  n >
  \max\left\{ \size{\Lambda}\, \size{\calM(\Lambda)}, \,
  2 \max(\Lambda) \right\}.
\] 
\end{lemma}

\begin{proof}
Following the argument in the proof of Lemma~\ref{lem:plan-b}, we now have
exact reconstruction of the cosine coefficients if and only if (instead of
\eqref{eq:claim})
\begin{align} \label{eq:claim-ck}
  Q_n^*((\cosb_{\bsk'}\circ\tent) \, \sqrt{2}^{|\bsk|_0} \cos(2\pi \bsk \cdot \bigcdot))
  \,=\,
  c_\bsk\,\delta_{\bsk',\bsk}
  \qquad
  \text{for all } \bsk, \bsk' \in \Lambda.
\end{align}
The case $\bsk\ne\bsk'$ is the same as in Lemma~\ref{lem:plan-b}. It
suffices to reconsider the case $\bsk = \bsk'$. Instead of separating out
the term $\bssigma(\bsk)=\bsk$ as in \eqref{eq:link-b2}, we apply the
character property \eqref{eq:character-property} for
$Q_n^*(e_{\pm(\bssigma(\bsk')-\bsk)})$ in \eqref{eq:link-b} with
$\bsk=\bsk'$ to arrive at
\begin{align*}
  Q_n^*((\cosb_{\bsk}\circ\tent) \, \sqrt{2}^{|\bsk|_0} \cos(2\pi\bsk \cdot \bigcdot))
  &\,=\,
  \frac{1}{2} \sum_{\satop{\bssigma\in\calS_\bsk}{\bssigma(\bsk)\cdot\bsz \equiv_n \bsk\cdot\bsz}} 2,
\end{align*}%
which is equal to $c_\bsk$ as required. The CBC result is proved in
Theorem~\ref{thm:cbc-plan-c} later.
\end{proof}

\subsection{Fast calculation of cosine coefficients and function values}

Here we can also make use of a one-dimensional fast Fourier transform to
map cosine coefficients to function values on the tent-transformed lattice
points and vice versa.

\begin{lemma} \label{lem:DCT-cos}
  Let $\bsz$ be a generating vector for an $n$-point rank-$1$ lattice satisfying
  the reconstruction property on an arbitrary index set $\Lambda \subset \NN_0^d$
  according to Lemma~\ref{lem:cos2} \textup{(}plan A\textup{)},
  Lemma~\ref{lem:plan-b} \textup{(}plan B\textup{)} or
  Lemma~\ref{lem:plan-c} \textup{(}plan C\textup{)}.
  For a function $f \in \calF^{\rm cos}_\Lambda$
  solely supported on $\Lambda$ we can compute
  \begin{align*}
  \begin{array}{ll}
    \mbox{coefficients from function values:} \\\hline
    \text{// prepare function value vector $\bsf \in \RR^n$} \\
    f_0 = f(\bszero)  \\
    \text{\textup{\texttt{for}} } i\in \{1,\ldots,\lfloor{n/2}\rfloor\}\text{\textup{\texttt{:}}} \\
    \quad f_i = f(\tent((i\bsz \bmod{n})/n))  \\
    \quad f_{n-i} = f_i \\[3mm]
    \text{// compute coefficient vector $\bsF \in \RR^n$} \\
    \bsF = \mathrm{FFT}(\bsf) \\[3mm]
    \text{// } \widehat{f}_\bsk \text{ is given by } \sqrt{2}^{|\bsk|_0} \, F_{(\bsk\cdot\bsz \bmod{n})} / c_\bsk \\
  \end{array}
  \;\left|\;
  \begin{array}{ll}
    \mbox{function values from coefficients:} \\\hline
    \text{// prepare coefficient vector $\bsF \in \RR^n$} \\
    \text{$\bsF = \bszero \in \RR^n$} \\
    \text{\textup{\texttt{for}} } \bsk \in \Lambda \text{\textup{\texttt{:}}}\\
    \quad \text{\textup{\texttt{for}} } \bssigma \in S_\bsk \text{\textup{\texttt{:}}}\\
    \quad\quad F_{(\bssigma(\bsk)\cdot\bsz \bmod{n})} = F_{(\bssigma(\bsk)\cdot\bsz \bmod{n})} + \widehat{f}_\bsk / \sqrt{2}^{|\bsk|_0} \\[3mm]
    \text{// compute function value vector $\bsf \in \RR^n$} \\
    \bsf = \mathrm{IFFT}(\bsF) \\[3mm]
    \text{// } f_i \text{ gives the value of } f(\tent((i\bsz \bmod{n})/n)) \\
  \end{array}\right.
  \end{align*}
  where $\bsf \in \RR^n$ is a vector containing function values with the symmetry
  $f_i = f_{n-i}$ and $\bsF \in \RR^n$ is a vector containing cosine coefficients
  with the symmetry $F_\kappa = F_{n-\kappa}$.
  Here $\mathrm{FFT}$ and $\mathrm{IFFT}$ are the one-dimensional fast Fourier
  transform and its inverse, respectively, with a normalization $1/n$ for $\mathrm{FFT}$
  and $1$ for $\mathrm{IFFT}$;
  both mappings have cost $\calO(n \log(n))$.
  For plan A and B we set $c_\bsk = 1$ for all $\bsk$.

  Alternatively, for even $n = 2m$ a length $m+1$ DCT-I can be used, while for odd $n = 2m-1$ a length $m$ DCT-V can be used.
  In this case the memory requirement and computational effort is halved
  (w.r.t.\ a real-to-real FFT).
\end{lemma}

\begin{proof}
  We show the result by using the calculation of Lemma~\ref{lem:plan-c} (plan C).
  Plan A and plan B are essentially the same with $c_\bsk = 1$.
  For all plans we have the option to use the inner product
  with respect to $\cos(2\pi\,i\,\bsk\cdot\bsz/n)$ as given in Lemma~\ref{lem:dot-innerproduct}
  since the reconstruction property is guaranteed
  by the conditions of plan C which is a subset of the conditions of plan B which in turn
  is a subset of the conditions of plan A.

Let $f_i := f(\tent((i\bsz \bmod{n})/n))$ for $i=0,\ldots,n-1$. In all
three plans we have for $\bsk\in\Lambda$,
  \begin{align*}
    \widehat{f}_\bsk
    =
    \frac{\sqrt{2}^{|\bsk|_0}}{c_\bsk} \,
    \frac1n \sum_{i=0}^{n-1}
      f_i\,
      \cos\Big(2\pi \, \frac{i\, \bsk\cdot\bsz}{n}\Big)
    &=
    \frac{\sqrt{2}^{|\bsk|_0}}{c_\bsk} \,
    \frac1n \sum_{i=0}^{n-1}
      f_i \,
      \left[ \cos\Big(2\pi \, \frac{i\, \bsk\cdot\bsz}{n}\Big)
       + \rmi \, \sin\Big(2\pi \, \frac{i\, \bsk\cdot\bsz}{n}\Big) \right]
    \\
    &=
    \frac{\sqrt{2}^{|\bsk|_0}}{c_\bsk} \,
    \underbrace{\frac1n \sum_{i=0}^{n-1}
      f_i\,
      \exp\Big(\!-2\pi\rmi \, \frac{i\, \bsk\cdot\bsz}{n}\Big)}_{=:\, F_{(\bsk\cdot\bsz\bmod{n})}}
      ,
  \end{align*}
which follows because of the symmetry of $f_i = f_{n-i}$ due to the tent
transform and the odd and even properties of the sine and cosine functions
respectively. The last expression is a scaled discrete Fourier transform
of length $n$ in terms of $i$ and $\kappa = \bsk \cdot \bsz \bmod{n}$.
This shows the calculation of coefficients from function values by
one-dimensional FFT.

Next we consider the evaluation of function values from coefficients. For
each $i=0,\ldots,n-1$, we have from \eqref{eq:tent-mapped-cosb} that
\begin{align*}
  f_i
  &\,=\, \sum_{\bsk\in\Lambda} \widehat{f}_\bsk \,\frac{1}{\sqrt{2}^{|\bsk|_0}}
  \sum_{\bssigma\in S_\bsk} \exp\Big(2\pi\rmi\, \frac{i\,\bssigma(\bsk)\cdot\bsz}{n}\Big)
  \,=\, \sum_{\kappa=0}^{n-1} \bigg(\underbrace{\sum_{\bsk\in\Lambda}
  \sum_{\satop{\bssigma\in S_\bsk}{\bssigma(\bsk)\cdot\bsz\equiv_n \kappa} }
  \widehat{f}_\bsk \,\frac{1}{\sqrt{2}^{|\bsk|_0}}}_{=:\,F_\kappa} \bigg)
  \exp\Big(2\pi\rmi\, \frac{i\,\kappa}{n}\Big),
\end{align*}
where it is unfortunately not possible to avoid considering all sign
changes of $\bsk$. For plan A and plan B there is only one $\kappa$
associated with each $\sigma(\bsk)$. For plan C it might occur that
different sign changes on the same $\bsk$ map to the same value of
$\kappa$, hence the summation in the algorithm to prepare the coefficient
vector $\bsF$.

Now we explain how to make use of DCT using symmetry.
  We have $f_i = f_{n-i}$ due to the symmetry of $\tent$.
  In the formula below, we will write $f_{n/2}$ which is to be interpreted in the
  way we just defined for even $n$, and to be considered equal to zero for odd $n$.
  Then, by making use of the symmetry, we have for general $n$ (odd or even)
  \begin{align*}
    \widehat{f}_\bsk
    &=
    \frac{\sqrt{2}^{|\bsk|_0}}{c_\bsk} \,
    \frac1n \bigg(
      f_0
      +
      2 \sum_{i=1}^{\lfloor (n-1)/2 \rfloor} f_i \cos\Big(2\pi \, \frac{i\,\kappa}{n}\Big)
      +
      f_{n/2} \, \cos(\pi \, \kappa)
    \bigg)
    ,
  \end{align*}
  where $f_{n/2} \, \cos(\pi \, \kappa)$ is only present for even $n$.
  Now for $n = 2m$ we have $\lfloor (n-1)/2 \rfloor = n/2 - 1 = m-1$ and we find
  \begin{align*}
    \widehat{f}_\bsk
    &=
    \frac{\sqrt{2}^{|\bsk|_0}}{c_\bsk} \,F_\kappa, \qquad\mbox{with}
    \qquad
    F_\kappa :=
    \frac1m \bigg(
      \frac12 \, f_0
      +
      \sum_{i=1}^{m-1} f_i\, \cos\Big(\pi \, \frac{i\,\kappa}{m}\Big)
      +
      \frac12 \, f_m \, \cos(\pi \, \kappa)
    \bigg)
    ,
  \end{align*}
  which is the formula for the one-dimensional DCT-I of length $m+1$ to turn the sequence $f_0,f_1,\ldots,f_m$ into the sequence $F_0,F_1,\ldots,F_m$.
  For odd $n = 2m-1$ we have $\lfloor (n-1)/2 \rfloor = m-1$ and we find
  \begin{align*}
    \widehat{f}_\bsk
    &=
    \frac{\sqrt{2}^{|\bsk|_0}}{c_\bsk} \,F_\kappa, \qquad\mbox{with}
    \qquad
    F_\kappa :=
    \frac1{2m-1} \bigg(
      f_0
      +
      2 \sum_{i=1}^{m-1} f_i \cos\Big(2 \pi \, \frac{i\,\kappa}{2m-1}\Big)
    \bigg)
    ,
  \end{align*}
which is the formula for the one-dimensional DCT-V of length $m$ to turn
the sequence $f_0,f_1$, \ldots, $f_{m-1}$ into the sequence
$F_0,F_1,\ldots,F_{m-1}$. For DCT-I and DCT-V see Martucci \cite{Mar94}
(definitions (A.1) and (A.5) therein).

Similarly, we have $F_\kappa = F_{n-\kappa}$ so we can write
\begin{align*}
  f_i
  &\,=\, F_0 + 2 \sum_{\kappa=1}^{\lfloor(n-1)/2\rfloor} F_\kappa\, \cos\Big(2\pi \frac{i\,\kappa}{n}\Big)
  + F_{n/2}\, \cos(\pi\,i),
\end{align*}
where $F_{n/2} := 0$ if $n$ is odd. Therefore DCT-I and DCT-V work in an
analogous way.
\end{proof}

We note that all coefficients calculated by the FFT are real because of the symmetry in the input, but in a typical implementation one would take the real part in case of a complex FFT routine to remove possible numerical noise. Alternatively one can make use of a special real
to real FFT implementation or use a specific implementation for the corresponding DCT.

Similarly to Remark~\ref{remark:Fourier-other-kappas}, one can also extend
the set $\Lambda$ such that $\sigma(\bsk)\cdot\bsz \bmod{n}$ covers as
many values as possible in $\ZZ_n$ for functions which are not solely
supported on $\Lambda$.

\section{Nonperiodic setting based on Chebyshev polynomials} \label{sec:che}

\subsection{Chebyshev series}

In the univariate case, the Chebyshev polynomials of the first kind for
$|x| \le 1$ can be written
\[ 
  T_k(x)
  \,=\,
  \cos(k \arccos(x))
  ,
  \qquad
  k = 0, 1, 2, \ldots,
  \qquad
  \text{for } x \in [-1, 1]
  .
\] 
We have orthogonality with respect to the Chebyshev weight $(\sqrt{1 -
x^2})^{-1}$:
\begin{align*}
  \int_{-1}^1 T_k(x) \, T_{k'}(x) \, \frac{\rd x}{\sqrt{1 - x^2}}
  &=
  \begin{cases}
    0, & \text{if $k \ne k'$}, \\
    \pi, & \text{if $k = k' = 0$}, \\
    \pi/2, & \text{if $k = k' \ne 0$} .
  \end{cases}
\end{align*}
To obtain an orthonormal basis on $[-1,1]$, we first normalize the measure
to 1 by adjusting the Chebyshev weight to $(\pi \sqrt{1-x^2})^{-1}$. Then
we define
\begin{align*}
  \cheb_k(x)
  :=
  \begin{cases}
    T_0(x) = 1,         & \text{if $k = 0$}, \\
    \sqrt{2} \, T_k(x), & \text{if $k = 1,2,\ldots$}\,,
  \end{cases}
\end{align*}
and for the multivariate case we define the tensor product basis functions
\begin{align} \label{eq:cheb}
  \cheb_\bsk(\bsx)
  :=
  \prod_{j=1}^d \cheb_{k_j}(x_j)
  =
  \sqrt{2}^{|\bsk|_0} \prod_{j=1}^d T_{k_j}(x_j),
\end{align}
where, as for cosine basis, $|\bsk|_0$ denotes the count of the nonzero
entries in the vector $\bsk$.

Let $\calF^{\rm Cheb}$ denote the space of real valued functions defined
on $[-1,1]^d$ with absolutely converging Chebyshev series:
$$
  \calF^{\rm Cheb}
  :=
  \Bigl\{ f \in L^2 \mathbin{\Big|} f : [-1,1]^d \to \RR,\;
  f(\bsx) = \sum_{\bsk \in \NN_0^d} \widehat{f}_\bsk \, \cheb_\bsk(\bsx)
  \text{ and } \sum_{\bsk \in \NN_0^d} |\widehat{f}_\bsk| < \infty \Bigr\},
$$
where $\widehat{f}_\bsk$ are the Chebyshev coefficients of $f$. We equip
$\calF^{\rm Cheb}$ with the weighted $L^2$ inner product
\[ 
 \langle f_1, f_2 \rangle_\mu
 \,:=\,
 \int_{[-1,1]^d} f_1(\bsx) \, f_2(\bsx) \, \mu(\rd\bsx),
 \qquad
 \mu(\rd\bsx) \,:=\, \frac{\rmd{\bsx}}{\prod_{j=1}^d \left(\pi \sqrt{1-x_j^2}\right)}.
\] 
The orthonormal Chebyshev basis functions satisfy
$\langle\cheb_\bsk,\cheb_{\bsk'}\rangle_\mu = \delta_{\bsk,\bsk'}$. The
Chebyshev coefficients are
  \begin{align*} 
    \widehat{f}_\bsk
    :=
    \langle f, \cheb_\bsk \rangle_\mu
    =
    \int_{[-1,1]^d} f(\bsx) \, \cheb_\bsk(\bsx) \, \mu(\rd\bsx).
\end{align*}

Given an arbitrary finite index set $\Lambda\subset\bbN_0^d$, we consider
the subspace $\calF^{\rm Cheb}_\Lambda$ of all functions whose Chebyshev
series is supported solely on $\Lambda$, i.e.,
\[ 
  \text{for } f \in \calF^{\rm Cheb}_\Lambda: \qquad
  f(\bsx)\,=\, \sum_{\bsk\in\Lambda} \widehat{f}_\bsk\,\cheb_\bsk(\bsx).
\] 

\subsection{Isomorphism with the cosine space via cosine transform}

The Chebyshev basis functions and the cosine basis functions are related
by the mapping
\[
  \cheb_\bsk(\bsx) = \cosb_\bsk(\bsx')
  \quad\iff\quad
  \arccos(\bsx) = \pi\,\bsx'
  \quad\iff\quad
  \bsx = \cos(\pi\,\bsx'),
\]
where the cosine function and its inverse are applied componentwise. This
provides an isomorphism between the Chebyshev setting and the cosine
space, with
\[
  f\in\calF^{\rm Cheb} \quad\iff\quad f_{\rm cos} \,:=\, f(\cos(\pi\,\bigcdot)) \,\in\calF^{\rm cos}.
\]
To get a sense of how this transformation works, it is informative to
consider the univariate case:
\begin{align*}
  \int_{-1}^1 f(x)\,\frac{\rmd x}{\pi\,\sqrt{1-x^2}}
  \,=\, \int_1^0 f(\cos(\pi x'))\,\frac{-\pi\,\sin (\pi x')\,\rmd x'}{\pi\,\sqrt{1-\cos^2(\pi x')}}
  \,=\, \int_0^1 f(\cos(\pi x'))\,\rmd x'.
\end{align*}

For the multivariate case we have the integral operator
\[
  I_\mu(f)
  \,:=\, \int_{[-1,1]^d} f(\bsx)\, \mu(\rmd\bsx)
  \,=\, \int_{[0,1]^d} f(\cos(\pi\,\bsx'))\,\rmd\bsx'
  \,=\, I(f(\cos(\pi\,\bigcdot))) \,=\, I(f_{\rm cos}),
\]
so that $I_\mu(\cheb_\bsk) = I(\cosb_\bsk)$,
$I_\mu(\cheb_\bsk\,\cheb_{\bsk'}) = I(\cosb_\bsk\,\cosb_{\bsk'})$, and
\[
  \langle f,\cheb_\bsk\rangle_\mu
  = \langle f(\cos(\pi\,\bigcdot)),\cosb_\bsk\rangle
  = \langle f_{\rm cos},\cosb_\bsk\rangle,
\]
i.e., the Chebyshev coefficients of $f$ are precisely the cosine
coefficients of $f_{\rm cos}$.

\subsection{Chebyshev coefficients by transformed rank-1 lattice rules}

The isomorphism between the spaces means that we can bring all results
from the cosine space over to the Chebyshev space. Noting the useful
property
\[ 
  \cos(\pi\,\tent(\bst)) \,=\, \cos(2\pi\,\bst),
\] 
we then arrive at the results for a \emph{tent-transformed and then
cosine-transformed lattice rule}, which is given explicitly by
\[
  Q_n(f) \,:=\,
  \frac{1}{n} \sum_{i=0}^{n-1} f\left(\cos\left(2\pi\frac{i\bsz\bmod n}{n}\right)\right)
  \,=\,
  \frac{1}{n} \sum_{i=0}^{n-1} f\left(\cos\left(2\pi\frac{i\bsz}{n}\right)\right)
  .
\]
Since $\cos(2\pi\,i\bsz/n) = \cos(2\pi\,(n-i)\bsz/n)$, the cubature points
double up and we can write
\[
  Q_n(f) \,=\,
  \begin{cases}
  \displaystyle
  \frac{f(\bsone)}{n}  + \frac{f(-\bsone)}{n}
  + \frac{2}{n} \sum_{i=1}^{n/2-1} f\left(\cos\left(2\pi\frac{i\bsz}{n}\right)\right)
  & \mbox{if $n$ is even}, \\
  \displaystyle
  \frac{f(\bsone)}{n}
  + \frac{2}{n} \sum_{i=1}^{(n-1)/2} f\left(\cos\left(2\pi\frac{i\bsz}{n}\right)\right)
  & \mbox{if $n$ is odd}.
  \end{cases}
\]
Thus the cubature rule can be computed with $\lfloor n/2 + 1 \rfloor$
function evaluations, where the cubature weight for $i=0$ and $i=n/2$ (if
$n$ is even) are $1/n$ and the others are $2/n$. In general there can be
further duplication of points. However, if $\gcd(z_j,n) = 1$ for at least
one $j=1,\ldots,d$, then all the $\lfloor n/2 + 1 \rfloor$ points are
distinct, see Lemma~\ref{lem:unique-points}.

For even $n=2m$ this point set has previously been called a ``Chebyshev
lattice'', see, e.g., \cite{CP11,PV15}, defined by $\bigl\{ \cos(\pi i\bsz
/ m) : i = 0, \ldots, m \bigr\}$. However, we prefer the interpretation as
a tent-transformed and then cosine-transformed lattice, since then we can
also use odd $n$, and the cubature weights are automatically correct
according to the multiplicity of the points.

We now state the analogous results to Lemmas~\ref{lem:cos1},
\ref{lem:cos2}, \ref{lem:plan-b}, \ref{lem:plan-c}, \ref{lem:DCT-cos} from
the cosine space.

\begin{lemma}[Integral exactness] \label{lem:che1}
Let $\Lambda \subset \NN_0^d$ be an arbitrary index set. A
tent-transformed and then cosine transformed lattice rule with $n$ points
and generating vector $\bsz$ integrates exactly (against the Chebyshev
density) all Chebyshev space functions $f \in \calF^{\rm Cheb}_\Lambda$
solely supported on $\Lambda$ if and only if
\[
  \bsh\cdot\bsz\not\equiv_n 0
  \qquad\mbox{for all }
  \bsh \in \calM(\Lambda)\nozero.
\]
Such a generating vector $\bsz$ can be constructed component-by-component
if $n$ is a prime satisfying
$$
  n > \max\left\{ \frac{\sizep{\calM(\Lambda)\nozero}}{2} + 1 , \,
  \max(\Lambda) \right\}
  .
$$
\end{lemma}

\begin{lemma}[Function reconstruction -- plan A] \label{lem:che2}
Let $\Lambda \subset \NN_0^d$ be an arbitrary index set. A lattice rule
$Q_n^*$ with $n$ points and generating vector $\bsz$ reconstructs exactly
the Chebyshev coefficients of all Chebyshev space functions $f \in
\calF^{\rm Cheb}_\Lambda$ solely supported on $\Lambda$, by
\begin{align*}
 \widehat{f}_\bsk \,=\,
 \widehat{f}_\bsk^a 
  &\,:=\,
  Q_n^*\left(f(\cos(2\pi\,\bigcdot)) \, (\cosb_\bsk\circ\tent)\right)
  \qquad
  \text{for all } \bsk \in \Lambda,
\end{align*}
if and only if
\[
  \bsh\cdot\bsz\not\equiv_n 0
  \qquad\mbox{for all }
  \bsh \in \calM(\Lambda) \oplus \calM(\Lambda)\nozero.
\]
Such a generating vector $\bsz$ can be constructed component-by-component
if $n$ is a prime satisfying
\[ 
  n > \max\left\{ \frac{\sizep{\calM(\Lambda)\oplus \calM(\Lambda)}+1}{2}, \,
  2 \max(\Lambda) \right\}
\] 
\end{lemma}

\begin{lemma}[Function reconstruction -- plan B] \label{lem:che-plan-b}
Let $\Lambda \subset \NN_0^d$ be an arbitrary index set. A lattice rule
$Q_n^*$ with $n$ points and generating vector $\bsz$ reconstructs exactly
the Chebyshev coefficients of all Chebyshev space functions
$f\in\calF^{\rm Cheb}_\Lambda$ solely supported on $\Lambda$, by
\begin{align*}
 \widehat{f}_\bsk \,=\,
 \widehat{f}_\bsk^b 
  &\,:=\,
  Q_n^*\left(f(\cos(2\pi\,\bigcdot)) \, \sqrt{2}^{|\bsk|_0} \cos(2\pi \bsk \cdot \bigcdot)\right)
  \qquad
  \text{for all } \bsk \in \Lambda,
\end{align*}
if and only if
\[ 
 \bssigma(\bsk')\cdot\bsz \not\equiv_n \bsk\cdot\bsz
 \qquad
 \mbox{for all } \bsk,\bsk' \in \Lambda, \, \bssigma\in\calS_{\bsk'},\, \bssigma(\bsk')\ne\bsk,
\] 
which is equivalent to
\[ 
 \bsh\cdot\bsz \not\equiv_n 0
 \qquad
 \mbox{for all } \bsh \in (\Lambda\oplus\calM(\Lambda))\nozero.
\] 
Such a generating vector $\bsz$ can be constructed component-by-component
if $n$ is a prime satisfying
\[ 
  n >
  \max\left\{ \sizep{\Lambda \oplus \calM(\Lambda)}, \,
  2 \max(\Lambda)\right\}.
\] 
\end{lemma}

\begin{lemma}[Function reconstruction -- plan C] \label{lem:che-plan-c}
Let $\Lambda \subset \NN_0^d$ be an arbitrary index set. A lattice rule
$Q_n^*$ with $n$ points and generating vector $\bsz$ reconstructs exactly
the Chebyshev coefficients of all Chebyshev space functions
$f\in\calF^{\rm Cheb}_\Lambda$ solely supported on $\Lambda$, by
\begin{align}
 \widehat{f}_\bsk \,=\,
 \widehat{f}_\bsk^c \nonumber 
  &\,:=\,
\frac{Q_n^*\big(f(\cos(2\pi\,\bigcdot)) \, \sqrt{2}^{|\bsk|_0} \cos(2\pi \bsk \cdot \bigcdot)\big)}{c_\bsk}, \quad\mbox{with} \\
  c_\bsk \label{eq:che-plan-c-ck}
  &\,:=\,
  \size{ \big\{ \bssigma\in\calS_\bsk \;:\;
  \bssigma(\bsk)\cdot\bsz  \equiv_n \bsk \cdot \bsz \big\}}
  \quad
  \text{for all } \bsk \in \Lambda,
\end{align}
if and only if
\[ 
  \bssigma(\bsk')\cdot\bsz \not\equiv_n \bsk\cdot\bsz
  \qquad\mbox{for all }
  \bsk,\bsk' \in \Lambda, \, \bssigma\in\calS_{\bsk'},\, \bsk\ne \bsk'.
\] 
Such a generating vector $\bsz$ can be constructed component-by-component
if $n$ is a prime satisfying
\[ 
  n >
  \max\left\{ \size{\Lambda}\, \size{\calM(\Lambda)}, \,
  2 \max(\Lambda)\right\}.
\] 
\end{lemma}

\begin{lemma} \label{lem:DCT-che}
We can use FFTs or DCTs to map Chebyshev coefficients to function values
on tent-transformed and then cosine-transformed lattice points, and the
other way round, for an $n$-point rank-$1$ lattice rule with generating
vector $\bsz$ satisfying the non-aliasing conditions of
Lemma~\ref{lem:che2} \textnormal{(}plan A\textnormal{)},
Lemma~\ref{lem:che-plan-b} \textnormal{(}plan B\textnormal{)} or
Lemma~\ref{lem:che-plan-c} \textnormal{(}plan C\textnormal{)} on an
arbitrary index set $\Lambda \subset \NN_0^d$ by replacing
$\tent(\bigcdot)$ by $\cos(\pi\,\tent(\bigcdot))$ in the statement of
Lemma~\ref{lem:DCT-cos}.
\end{lemma}

Lemma~\ref{lem:che-plan-c} with even $n=2m$ in combination with the DCT-I
in Lemma~\ref{lem:DCT-che} is essentially the approach in Potts \& Volkmer
\cite{PV15}.

\section{The CBC construction} \label{sec:cbc}

\subsection{Induction proof for the component-by-component construction}

Let $\bbZ_n^* := \{1,\ldots,n-1\}$ for $n$ prime. The necessary and
sufficient conditions on the lattice rule generating vector
$\bsz\in(\bbZ_n^*)^d$ for integral exactness and function reconstruction
in most cases boil down to the same generic form of verifying for a given
index set $\calA\subset\bbZ^d$ that
\begin{align}\label{eq:generic}
     \bsh \cdot \bsz \not\equiv_n 0
     \qquad\mbox{for all}\quad \bsh \in \calA\nozero.
\end{align}
The following theorem justifies a generic component-by-component (CBC)
algorithm to find a $\bsz$ satisfying this condition. The inductive
argument needs to work with projections of the index set $\calA$ down to
the lower coordinates. We consider two definitions for the projections
since each has its advantages:
\begin{align}
  \label{eq:set-s0}
  \mbox{either}\qquad
  \calA_s &\,:=\, \{ \bsh \in \bbZ^s \,:\, (\bsh,\bszero) \in\calA \}
  &\mbox{for } s=1,\ldots,d,
  \\
  \label{eq:set-s}
  \mbox{or}\qquad
  \calA_s &\,:=\,
  \{ \bsh\in\bbZ^s \,:\, (\bsh,\bsh_\perp) \in\calA \,\mbox{ for some }\, \bsh_\perp\in\bbZ^{d-s}\}
  &\mbox{for } s=1,\ldots,d.
\end{align}
Both definitions yield $\calA_d = \calA$. The definition \eqref{eq:set-s0}
includes only the indices whose higher components are zero; we shall refer
to this as the `zero' projection. The definition \eqref{eq:set-s} includes
all indices obtained by truncating the original indices; we shall refer to
this as the `full' projection. If the index set $\calA$ is downward closed
then they are the same; otherwise \eqref{eq:set-s0} is a subset of
\eqref{eq:set-s}. The full projection \eqref{eq:set-s} was used in
\cite{CKN10,Kam13,Kam14}; the zero projection \eqref{eq:set-s0} is new in
this paper.

The condition $n > \max(\calA)$ in the theorem guarantees that the
components of $\bsh \in \calA$ all satisfy $|h_j|<n$, and thus $h_j
\equiv_n 0$ if and only if $h_j= 0$. This condition can be replaced by the
direct assumption that there is no $\bsh\in\calA$ with a nonzero component
$h_j$ that is a multiple of $n$.

\begin{theorem}\label{thm:cbc}
Let $\calA\subset\bbZ^d$ be an arbitrary index set, and let $n$ be a prime
number satisfying
\begin{equation} \label{eq:cond-n}
  n > \max \Big\{\frac{\sizep{\calA\nozero}}{\kappa} + 1\;,\; \max(\calA) \Big\},
\end{equation}
with $\kappa = 2$ if $\calA$ is centrally symmetric and $\kappa = 1$
otherwise. Define the projections $\calA_s$ by \eqref{eq:set-s0} or
\eqref{eq:set-s}. Then a generating vector $\bsz^*=(z_1,\ldots,z_d) \in
(\ZZ^*_n)^d$ can be constructed component-by-component such that for all
$s = 1, \ldots, d$ and $\bsz = (z_1, \ldots, z_s)$ we have
  \begin{align}\label{eq:good-z}
     \bsh \cdot \bsz \not\equiv_n 0
     \qquad\mbox{for all}\quad \bsh \in \calA_s\!\setminus\!\{\bszero\}.
  \end{align}
\end{theorem}

\begin{proof}
The result for centrally symmetric and downward closed index sets (e.g.\
hyperbolic cross) or more general index sets with the full projection
\eqref{eq:set-s} has been proved in \cite{CKN10,Kam13,Kam14}. So we focus
on proving the general result with the zero projection \eqref{eq:set-s0}.

The proof is by induction on $s$. We will attempt to derive the condition
\eqref{eq:cond-n} rather than assuming it from the beginning.

For $s=1$, the condition $h_1 z_1\not\equiv_n 0$ holds for all $z_1 \in
\ZZ^*_n$ if $h_1\not\equiv_n 0$, and fails for all $z_1$ if $h_1\equiv_n
0$. To avoid the latter scenario we assume that $n>|h_1|$ always holds.

Suppose we already obtained the generating vector $\bsz \in
(\ZZ^*_n)^{s-1}$ satisfying \eqref{eq:good-z} for $\calA_{s-1}$. For each
$(\bsh, h_s) \in \calA_s\nozero$, we will eliminate any `bad'
$z_s\in\bbZ_n^*$ that satisfies
\begin{align}\label{eq:bad-z}
     (\bsh, h_s) \cdot (\bsz, z_s)
     \equiv_n
     \bsh\cdot\bsz + h_s z_s
     \equiv_n
     0
     \quad\Leftrightarrow\quad
     h_s z_s
     \equiv_n - \bsh\cdot\bsz.
\end{align}
We stress that $(\bsh, h_s) \in \calA_s$ does not imply $\bsh \in
\calA_{s-1}$ under the zero projection \eqref{eq:set-s0}. Depending on the
value of $h_s$ we have the following scenarios:
\begin{enumerate}
  \item If $h_s = 0$ then $\bsh \in \calA_{s-1}\nozero$; in turn the
      induction hypotheses \eqref{eq:good-z} for $\calA_{s-1}$
      guarantees that $\bsh\cdot\bsz \not\equiv_n 0$ and so
      \eqref{eq:bad-z} has no solution for~$z_s$. There are
      $\sizep{\calA_{s-1}\nozero}$ such cases.

  \item If $h_s\ne 0$ but $h_s \equiv_n 0$, then \eqref{eq:bad-z} has
      no solution for~$z_s$ if $\bsh\cdot\bsz \not\equiv_n 0$, or it
      holds for all~$z_s$ if $\bsh\cdot\bsz \equiv_n 0$. To avoid the
      latter scenario we assume that $n > |h_s|$ always holds.

  \item If $h_s \not\equiv_n 0$ then, since $n$ is prime,
      \eqref{eq:bad-z}~has a unique solution for $z_s \in \ZZ^*_n$ if
      $\bsh\cdot\bsz \not\equiv_n 0$, or has no solution for $z_s \in
      \ZZ^*_n$ if $\bsh\cdot\bsz \equiv_n 0$. The latter scenario
      includes $\bsh=\bszero$ so there are at least
      $\size{\calA_s^\dag}$ such cases, where $\calA_s^\dag := \{
      h_s\in\bbZ : h_s \not\equiv_n 0 \mbox{ and } (\bszero, h_s) \in
      \calA_s\}$.
\end{enumerate}
Thus, provided that $n > |h_s|$, there is at most one bad $z_s$ to be
eliminated for each $(\bsh, h_s) \in \calA_s\nozero$, and the total number
of bad $z_s$ we eliminate is at most $\sizep{\calA_s\nozero} -
\sizep{\calA_{s-1}\nozero} - \size{\calA_s^\dag}$.

Hence, provided additionally that $\size{\ZZ^*_n} = n-1 > \size{\calA_s} -
\size{\calA_{s-1}} - \size{\calA_s^\dag}$, there is always a `good' $z_s$
remaining such that $(\bsz,z_s)$ will satisfy~\eqref{eq:good-z} for
$\calA_s$. Moreover if $\calA$ is centrally symmetric, then all $\calA_s$
are centrally symmetric, and both $(\bsh', h_s)$ and $(-\bsh, -h_s)$ will
eliminate the same $z_s$ if a solution for~\eqref{eq:bad-z} exists. It
then suffices to demand that $n-1 > (\size{\calA_s} - \size{\calA_{s-1}} -
\size{\calA_s^\dag})/2$. By induction, to ensure that a good
$\bsz^*\in(\bbZ_n^*)^d$ exists, it suffices to assume that
\[
  n > \max
\Big\{
\max_{s=2,\ldots,d}
    \frac{\size{\calA_s} - \size{\calA_{s-1}} - \size{\calA_s^\dag}}{\kappa} + 1\;,\;
  \max_{s=1,\dots,d} \max_{\bsh\in\calA_s} |h_s|
  \Big\},
\]
with $\kappa = 2$ if $\calA$ is centrally symmetric and $\kappa = 1$
otherwise. This leads to the simplified condition on $n$ in the theorem.

Now for completeness we discuss briefly the case for the full projection
\eqref{eq:set-s}. The proof is almost identical to the case for the zero
projection but is slightly simpler. The subtle difference is that for each
$(\bsh,h_s)\in\calA_s$ we now have $\bsh\in\calA_{s-1}$ regardless of the
value of $h_s$, and the induction hypothesis \eqref{eq:good-z} for
$\calA_{s-1}$ guarantees that $\bsh\cdot\bsz\not\equiv_n 0$ if $h_s\ne 0$,
thus simplifying the second and third scenarios above.
\end{proof}

\subsection{Algorithmic aspects of the CBC construction}

\begin{remark} \label{rem:cbc}
For the projections $\calA_s$ defined by either \eqref{eq:set-s0} or
\eqref{eq:set-s}, Theorem~\ref{thm:cbc} and its proof justify two
different approaches to carry out the component-by-component construction:
\begin{itemize}
\item Brute force approach: At step $s$, we search through
    $z_s\in\ZZ^*_n$ until we find one that satisfies \eqref{eq:good-z}
    for all $\bsh\in\calA_s\nozero$. The cost is $\calO(n_{\rm
    fail}\,\size{\calA_s})$, where $n_{\rm fail}$ is the number of
    different $z_s$ that was checked. So the cost is at worst
    $\calO(n\,\size{\calA_s})$, leading to a total cost of
    $\calO(d\,n\,\size{\calA})$.
\item Elimination approach: At step $s$, we loop through every
    $\bsh\in\calA_s\nozero$ and eliminate the corresponding
    $z_s\in\ZZ^*_n$ that fails \eqref{eq:good-z}, if any. Then we take
    any remaining $z_s$. The cost is only $\calO(\size{\calA_s})$,
    leading to a total cost of $\calO(d\,\size{\calA})$.
\end{itemize}
In both approaches we have left out the $\calO(s)$ factor in step $s$ that
arises from the evaluation of dot products; this is valid if we store and
update the values of $\bsh\cdot\bsz$ for all $\bsh$ in each step. In the
elimination approach we mark the bad choices of $z_s$ in an array of
length $n-1$ with pointers linking the previous and next good choices of
$z_s$, so that it is $\calO(1)$ cost to obtain a good $z_s$ at the end.

The two approaches may be used for different steps in the algorithm if it
is advantageous to mix them. Both approaches are guaranteed to succeed
provided $n$ is sufficiently large, see \eqref{eq:cond-n}. We can run the
algorithm with smaller values of $n$ (or even composite values of $n$ in
the brute force approach), and be prepared to increase $n$ when the
algorithm fails. Once a $\bsz$ is found, we can systematically test and
reduce the value of~$n$ by verifying whether \eqref{eq:good-z} still holds
for $\calA$.

In general the zero projections \eqref{eq:set-s0} are subsets of the full
projections \eqref{eq:set-s}, and consequently the condition
\eqref{eq:good-z} is weaker and faster to check for \eqref{eq:set-s0} than
for \eqref{eq:set-s}. There are also algorithmic advantages in the data
structure for iterating the sets based on \eqref{eq:set-s0}, namely, that
the indices can be ordered according to the number of zeros at the end.
\end{remark}

We now apply Theorem~\ref{thm:cbc} to the situation where the input set
$\calA$ is a difference set, i.e., $\calA = \Lambda\ominus\Lambda$. Then
the condition \eqref{eq:good-z} is now explicitly given by
\begin{align}\label{eq:diff-z1}
     \bsh \cdot \bsz \not\equiv_n 0
     \qquad\mbox{for all}\quad \bsh \in (\Lambda\ominus\Lambda)_s \setminus\{\bszero\},
\end{align}
where $(\Lambda\ominus\Lambda)_s$ denote the projection of the difference
set $\Lambda\ominus\Lambda$ defined according to \eqref{eq:set-s0} or
\eqref{eq:set-s}. Since $\sizep{\Lambda\ominus\Lambda}\le
(\size{\Lambda})^2$, Remark~\ref{rem:cbc} indicates that the cost of CBC
construction for $\calA = \Lambda\ominus\Lambda$ is
$\calO(d\,n\,(\size{\Lambda})^2)$ or $\calO(d\,(\size{\Lambda})^2)$ for
the two approaches, respectively.

Similarly, for all sign changes on an index set $\Lambda$ we have
$\size{\calM(\Lambda)}\le 2^d\,\size{\Lambda}$ and therefore the cost of
CBC construction for $\calA = \calM(\Lambda) \oplus \calM(\Lambda)$ is
$\calO(d\,n\,2^{2d}\,(\size{\Lambda})^2)$ or
$\calO(d\,2^{2d}\,(\size{\Lambda})^2)$; the cost for $\calA = \Lambda
\oplus \calM(\Lambda)$ is $\calO(d\,n\,2^d\,(\size{\Lambda})^2)$ or
$\calO(d\,2^d\,(\size{\Lambda})^2)$.

Hence, we can apply Theorem~\ref{thm:cbc} and Remark~\ref{rem:cbc} to the
Fourier space in Lemmas~\ref{lem:fou1} and~\ref{lem:fou2}, noting that the
difference set $\Lambda\ominus\Lambda$ is always centrally symmetric and
contains the zero vector. Analogously, we can apply Theorem~\ref{thm:cbc}
and Remark~\ref{rem:cbc} to the cosine space in Lemmas~\ref{lem:cos1}
and~\ref{lem:cos2}, as well as Lemma~\ref{lem:plan-b} -- plan B, and
correspondingly, to the Chebyshev space in Lemmas~\ref{lem:che1},
\ref{lem:che2}, and~\ref{lem:che-plan-b}.

However, plan C for the cosine space and Chebyshev space, see
Lemmas~\ref{lem:plan-c} and~\ref{lem:che-plan-c}, respectively, cannot be
formulated in the same generic form \eqref{eq:generic}. So we will need to
develop a separate justification for it. We will return to this later in
Subsection~\ref{sec:cbc-plan-c}.

\subsection{Smart lookup for the brute force approach with full projection}

When $\calA_s$ are full projections \eqref{eq:set-s}, we have the
important property that the projection of the difference set equals the
difference set of the projections, i.e.,
\begin{equation}
  (\Lambda\ominus\Lambda)_s = \Lambda_s\ominus\Lambda_s
  \qquad\mbox{for the full projection \eqref{eq:set-s}.} \label{eq:key}
\end{equation}
Thus the condition \eqref{eq:diff-z1} becomes
\[
      \bsh \cdot \bsz \not\equiv_n 0
     \qquad\mbox{for all}\quad \bsh \in (\Lambda_s\ominus\Lambda_s)\nozero,
\]
which is equivalent to
\begin{align}\label{eq:diff-z2}
  \bsh \cdot \bsz \not\equiv_n \bsh' \cdot \bsz
  \qquad
  \text{for all}\quad \bsh,\bsh' \in \Lambda_s \text{ with } \bsh \ne \bsh'
  .
\end{align}
In other words, every dot product needs to have a unique value.

The following code snippet shows that it is possible to verify condition
\eqref{eq:diff-z2} for a given $\bsz$ with cost $\calO(\size{\Lambda_s})$
rather than $\calO((\size{\Lambda_s})^2)$, by marking a bit string of
length $n$ for the values of dot product modulo $n$ that have occurred.
\begin{lstlisting}[mathescape]
 // Fourier space
 // INPUT: $\bsz$, $n$ and $\Lambda_s$
 // VERIFY: $\bsh \cdot \bsz \not\equiv_n \bsh' \cdot \bsz$ for all $\bsh,\bsh' \in \Lambda_s$, $\bsh \ne \bsh'$
 // COST: $\calO(\size{\Lambda_s})$

 $S = 0$
 for $\bsh \in \Lambda_s$:
     $\alpha = \bsh \cdot \bsz \bmod{n}$
     if $S[\alpha]=1$: return FALSE
     $S[\alpha] = 1$
 return TRUE
\end{lstlisting}
Consequently, we can reduce the cost of the brute force CBC construction
for $\calA = \Lambda\ominus\Lambda$ from $\calO(d\,n\,(\size{\Lambda})^2)$
to $\calO(d\,n\,\size{\Lambda})$. We shall refer to this as the ``smart
lookup'' trick.

We stress once again that \eqref{eq:key} only holds when we have the full
projection \eqref{eq:set-s}. Under the zero projection \eqref{eq:set-s0}
we would have in general $(\Lambda\ominus\Lambda)_s \supseteq
\Lambda_s\ominus\Lambda_s$; in this case the alternative formulation
\eqref{eq:diff-z2} would miss out on some indices.

Similar reduction in cost can be achieved for $\calA =
\calM(\Lambda)\oplus \calM(\Lambda)$ as we show in the code snippet below.
\begin{lstlisting}[mathescape]
 // Cosine space and Chebyshev space -- plan A
 // INPUT: $\bsz$, $n$ and $\Lambda_s$
 // VERIFY: $\bssigma'(\bsk') \cdot \bsz \not\equiv_n \bssigma(\bsk) \cdot \bsz$ for all $\bsk,\bsk' \in \Lambda_s$, $\bssigma\in\calS_\bsk$, $\bssigma'\in\calS_{\bsk'}$, $\bssigma'(\bsk') \ne \bssigma(\bsk)$
 // COST: $\calO(\size{\calM(\Lambda_s}))$

 $S = 0$
 for $\bsk \in \Lambda_s$:
     for $\bsh \in \{ \bssigma(\bsk) : \bssigma \in \calS_\bsk \}$:
         $\alpha = \bsh \cdot \bsz \bmod{n}$
         if $S[\alpha]=1$: return FALSE
         $S[\alpha]=1$
 return TRUE
\end{lstlisting}

We can save on half of the calculations, since if $\bsk \ne \bszero$ we
can fix one of the signs for a non-zero element of $\bsk$ to get half of
the sign changes and multiply by $-1$ to get the other half as shown
below.

\begin{lstlisting}[mathescape]
 // Cosine space and Chebyshev space -- plan A -- halved
 // INPUT: $\bsz$, $n$ and $\Lambda_s$
 // VERIFY: $\bssigma'(\bsk') \cdot \bsz \not\equiv_n \bssigma(\bsk) \cdot \bsz$ for all $\bsk,\bsk' \in \Lambda_s$, $\bssigma\in\calS_\bsk$, $\bssigma'\in\calS_{\bsk'}$, $\bssigma'(\bsk') \ne \bssigma(\bsk)$
 // COST: $\calO(\size{\calM(\Lambda_s})/2)$

 $S = 0$
 if $\bszero \in \Lambda_s$: $S[0] =1$
 for $\bsk \in \Lambda_s \setminus\{\bszero\}$:
     $i = \min\{ j : k_j \ne 0 \}$
     for $\bsh \in \{ \bssigma(\bsk) : \bssigma \in \calS_\bsk, \sigma_i = +1 \}$:
         $\alpha = \bsh \cdot \bsz \bmod{n}$        // note: $n - \alpha = -\bsh \cdot \bsz \bmod{n}$
         if $S[\alpha]=1$: return FALSE
         $S[\alpha]=1$
         if $S[n - \alpha]=1$: return FALSE   // note: it could happen that $n-\alpha = \alpha$
         $S[n-\alpha]=1$
 return TRUE
\end{lstlisting}

We can achieve a similar reduction in cost for $\calA = \Lambda\oplus
\calM(\Lambda)$, but this is more complicated because we need to
distinguish between the dot products coming from the original indices and
the dot products coming from sign changes of the indices. We do this by
keeping two bit strings of length $n$ as shown in the code snippet below:
$S_1$ marks the original dot products, while $S_2$ marks the dot products
from all sign changes thus including $S_1$. (We can also half the cost as
above but we do not include that here.)

\begin{lstlisting}[mathescape]
 // Cosine space and Chebyshev space -- plan B
 // INPUT: $\bsz$, $n$ and $\Lambda_s$
 // VERIFY: $\bssigma(\bsk') \cdot \bsz \not\equiv_n \bsk \cdot \bsz$ for all $\bsk,\bsk' \in \Lambda_s$, $\bssigma \in\calS_{\bsk'}$, $\bssigma(\bsk') \ne \bsk$
 // COST: $\calO(\size{\calM(\Lambda_s}))$

 $S_1 = 0$
 $S_2 = 0$
 for $\bsk \in \Lambda_s$:
     $\alpha = \bsk \cdot \bsz \bmod{n}$
     if $S_2[\alpha]=1$: return FALSE     // note: the value of $S_1[\alpha]$ is also checked since $S_1 \subseteq S_2$
     $S_2[\alpha]=1$
     $S_1[\alpha]=1$
     for $\bsh \in \{ \bssigma(\bsk) : \bssigma \in \calS_\bsk \}$ with $\bsh\ne\bsk$:
         $\alpha' = \bsh \cdot \bsz \bmod{n}$
         if $S_1[\alpha']=1$: return FALSE
         $S_2[\alpha']=1$        // note: it does not matter if $S_2[\alpha']$ is already set
 return TRUE
\end{lstlisting}

The previous algorithm can now be modified to allow self-aliasing and to
keep track of the constant $c_\bsk$, see Lemma~\ref{lem:plan-c} for cosine
space and Lemma~\ref{lem:che-plan-c} for Chebyshev space.

\begin{lstlisting}[mathescape]
 // Cosine space and Chebyshev space -- plan C
 // INPUT: $\bsz$, $n$ and $\Lambda_s$
 // VERIFY: $\bssigma(\bsk') \cdot \bsz \not\equiv_n \bsk \cdot \bsz$ for all $\bsk,\bsk' \in \Lambda_s$, $\bssigma \in \calS_{\bsk'}$, $\bsk' \ne \bsk$
 // COST: $\calO(\size{\calM(\Lambda_s}))$

 $S_1 = 0$
 $S_2 = 0$
 for $\bsk \in \Lambda_s$:
     $\alpha = \bsk \cdot \bsz \bmod{n}$
     if $S_2[\alpha] = 1$: return FALSE
     $S_2[\alpha]=1$
     $c_\bsk = 1$
     for $\bsh \in \{ \bssigma(\bsk) : \bssigma \in \calS_\bsk \}$ with $\bsh\ne\bsk$:
         $\alpha' = \bsh \cdot \bsz \bmod{n}$
         if $\alpha' = \alpha$: $c_\bsk$ += $1$
         if $S_1[\alpha']=1$: return FALSE
         $S_2[\alpha']=1$
     $S_1[\alpha]=1$
 return TRUE with $\{c_\bsk : \bsk\in\Lambda_s\}$
\end{lstlisting}

The crucial difference between the last two code snippets is that in plan
B the bit $S_1[\alpha]$ is marked before the dot products $\alpha'$ from
the sign changes are checked against $S_1$, thus not allowing
$\alpha'=\alpha$ (no aliasing), while in plan C the bit $S_1[\alpha]$ is
marked only after all $\alpha'$ have been checked against $S_1$, thus
allowing $\alpha' = \alpha$ and indeed counts the number of times this
occurs in $c_\bsk$ (self-aliasing).

We summarize this subsection in the following remark.

\begin{remark} \label{rem:cbc2}
The cost for brute force CBC in step $s$ with full projection
\eqref{eq:set-s} and smart lookup~is
\[
  \calO(n_{\rm fail}\,\size{\Lambda}_s), \qquad
  \calO(n_{\rm fail}\,\size{\calM(\Lambda}_s)), \qquad
  \calO(n_{\rm fail}\,\size{\calM(\Lambda}_s)),
\]
for the index sets $\calA = \Lambda\ominus\Lambda$, $\calA =
\calM(\Lambda)\oplus\calM(\Lambda)$ and $\calA =
\Lambda\oplus\calM(\Lambda)$, respectively.
\end{remark}

Similar strategies have been implemented in the sparseFFTr1l software
library of Toni Volkmer \cite{Volkmer}.

\subsection{Mixed CBC construction}

Combining Remark~\ref{rem:cbc2} with Remark~\ref{rem:cbc}, we see that
there is advantage in mixing the two different approaches. As long as
$n_{\rm fail}$ remains small it is advantageous to follow the brute force
approach with full projection \eqref{eq:set-s} and smart lookup. We
anticipate this to be the case for the initial dimensions.

Starting from $z_1 = 1$, at step $s$ we begin our brute force search with
the value $z_s = z_{s-1} + 1$. If this $z_s$ fails then we increment again
by $1$ and do this repeatedly (if $n-1$ is reached then we continue from
$1$) until a valid $z_s$ is found, while keeping a count on $n_{\rm
fail}$.
Then gradually as the dimension increases and as we run out of choices, we
expect the value of $n_{\rm fail}$ to increase until at some point the
balance tips over the other way and it becomes cheaper to follow the
elimination approach. From then on we switch over to the elimination
approach in the generic formulation \eqref{eq:generic} with the zero
projection \eqref{eq:set-s0} so that the sets are smaller (except for the
case of plan C which we discuss in the next subsection).

We summarize our results for the different spaces in Table~\ref{tab:cbc}.

\begin{table} [ht]
\caption{Summary of CBC algorithms for function reconstruction}
\label{tab:cbc} \small\vspace{0.3cm}
\begin{tabular}{|@{\,}c@{\,}|}
 \hline \\
 {\setlength{\extrarowheight}{3pt}
\begin{tabular}{@{\!}c@{\,}|c|c|c|c|}
 \cline{2-5}
 & Fourier space & \multicolumn{3}{c|}{Cosine space and Chebyshev space} \\
 \cline{3-5}
 & & Plan A & Plan B & Plan C \\
 \cline{2-5}
 (a) & $\bsh\cdot\bsz\not\equiv_n 0$ \quad for all
     & $\bsh\cdot\bsz\not\equiv_n 0$ \quad for all
     & $\bsh\cdot\bsz\not\equiv_n 0$ \quad for all
     & NA
     \\
     & $\bsh\in (\Lambda\ominus\Lambda)_s\!\setminus\!\{\bszero\}$
     & $\bsh\in (\calM(\Lambda)\!\oplus\!\calM(\Lambda))_s\!\setminus\!\{\bszero\}$
     & $\bsh\in (\Lambda\oplus\!\calM(\Lambda))_s\!\setminus\!\{\bszero\}$
     &
     \\
 \cline{2-5}
 (b) & $\bsh\cdot\bsz\not\equiv_n \bsh'\cdot\bsz$
     & $\bssigma'(\bsk')\cdot\bsz\not\equiv_n \bssigma(\bsk)\cdot\bsz$
     & $\bssigma(\bsk')\cdot\bsz\not\equiv_n \bsk\cdot\bsz$
     & $\bssigma(\bsk')\cdot\bsz\not\equiv_n \bsk\cdot\bsz$
     \\
     & for all $\bsh,\bsh'\in \Lambda_s,$
     & for all $\bsk,\bsk'\in \Lambda_s,$
     & for all $\bsk,\bsk'\in \Lambda_s,$
     & for all $\bsk,\bsk'\in \Lambda_s,$
     \\
     & $\bsh\ne\bsh'$
     & $\bssigma\in\calS_\bsk$, $\bssigma'\in\calS_{\bsk'},$
     & $\bssigma\in\calS_{\bsk'},$
     & $\bssigma\in\calS_{\bsk'},$
     \\
     &
     & $\bssigma'(\bsk')\ne\bssigma(\bsk)$
     & $\bssigma(\bsk')\ne\bsk$
     & $\bsk\ne\bsk'$
     \\
 \cline{2-5}
 (c) & $n \sim \size{(\Lambda\ominus\Lambda)_s}$
     & $n \sim \size{(\calM(\Lambda)\!\oplus\!\calM(\Lambda))_s}$
     & $n \sim \size{(\Lambda\!\oplus\!\calM(\Lambda))_s}$
     & $n \sim \size{\Lambda_s}\, \size{\calM(\Lambda_s)}$
     \\
 \cline{2-5}
 (d) & \cellcolor{green!25} $\size{(\Lambda\ominus\Lambda)_s}$
     & \cellcolor{green!25} $\size{(\calM(\Lambda)\!\oplus\!\calM(\Lambda))_s}$
     & \cellcolor{green!25} $\size{(\Lambda\!\oplus\!\calM(\Lambda))_s}$
     & NA
 \\
 \cline{2-5}
 (e) & $n_{\rm fail}\,\size{(\Lambda\ominus\Lambda)_s}$
     & $n_{\rm fail}\,\size{(\calM(\Lambda)\!\oplus\!\calM(\Lambda))_s}$
     & $n_{\rm fail}\,\size{(\Lambda\!\oplus\!\calM(\Lambda))_s}$
     & NA
 \\
 \cline{2-5}
 (f) & $(\size{\Lambda}_s)^2$
     & $(\size{\calM(\Lambda_s)})^2$
     & $\size{\Lambda_s}\,\size{\calM(\Lambda_s)}$
     & \cellcolor{yellow!25} $\size{\Lambda_s}\,\size{\calM(\Lambda_s)}$
 \\
 \cline{2-5}
 (g) & $n_{\rm fail}\,(\size{\Lambda}_s)^2$
     & $n_{\rm fail}\,(\size{\calM(\Lambda_s)})^2$
     & \multicolumn{2}{c|}{$n_{\rm fail}\,\size{\Lambda_s}\,\size{\calM(\Lambda_s)}$ }
 \\
 \cline{2-5}
 (h) & \cellcolor{blue!25} $n_{\rm fail}\,\size{\Lambda_s}$
     & \cellcolor{blue!25} $n_{\rm fail}\,\size{\calM(\Lambda_s)}$
     & \multicolumn{2}{c|}{\cellcolor{blue!25} $n_{\rm fail}\,\size{\calM(\Lambda_s)}$}
 \\
 \cline{2-5}
 (i) & $n_{\rm fail} \sim \size{\Lambda_s}$
     & $n_{\rm fail} \sim \size{\calM(\Lambda_s)}$
     & \multicolumn{2}{c|}{$n_{\rm fail} \sim \size{\calM(\Lambda_s)}$  }
     \\
 \cline{2-5}
 (j) & \cellcolor{red!25} $d\,(\size{\Lambda})^2$
     & \cellcolor{red!25} $d\,(\size{\calM(\Lambda)})^2$
     & \multicolumn{2}{c|}{\cellcolor{red!25} $d\,\size{\Lambda}\,\size{\calM(\Lambda)}$  }
 \\
 \cline{2-5}
 \end{tabular}
 }
 \\
 \begin{tabular}{c@{\;\;}l}
 &\\
 (a) & Standard formulation of the reconstruction condition at step $s$ with full/zero projection \\
 (b) & Equivalent formulation of the reconstruction condition at step $s$ with full projection \\
 (c) & Required size of $n$ to guarantee success at step $s$ (also need $n$ to cover spread of index set) \\
 (d) & \cellcolor{green!25}Cost of elimination approach at step $s$ based on (a) with full/zero projection \\
 (e) & Cost of brute force approach at step $s$ based on (a) with full/zero projection \\
 (f) & \cellcolor{yellow!25}Cost of elimination approach at step $s$ based on (b) with full projection \\
 (g) & Cost of brute force approach at step $s$ based on (b) with full projection \\
 (h) & \cellcolor{blue!25}Cost of brute force approach at step $s$ based on (b) with full projection and smart lookup \\
 (i) & Switching point on $n_{\rm fail}$ from brute force (h) to elimination (d)/(f) \\
 (j) & \cellcolor{red!25}Total cost of mixed CBC: brute force until $n_{\rm fail}$ reaches switching point then elimination \\
 \end{tabular}
 \\ \\
 \hline
\end{tabular}
\end{table}

\subsection{A new CBC proof for plan C} \label{sec:cbc-plan-c}

Recall that the condition \eqref{eq:plan-c-iff} is weaker than the
condition \eqref{eq:plan-b-iff}, which is in turn equivalent to
\eqref{eq:plan-b-iff-2}. Thus when we have the full projection
\eqref{eq:set-s}, the condition on $n$ in Theorem~\ref{thm:cbc} guarantees
the existence of $z_s$ with the required property in step $s$. However, to
prove that the CBC construction can find this vector, we need a new CBC
proof.

\begin{theorem} \label{thm:cbc-plan-c}
Let $\Lambda\subset\bbN^d_0$ be an arbitrary index set, and let $n$ be a
prime number satisfying
\begin{equation} \label{eq:cond-n-c}
  n > \max \Big\{\size{\Lambda}\,\size{\calM(\Lambda)}\;,\; 2\max(\Lambda) \Big\}.
\end{equation}
Define $\Lambda_s$ to be the full projection of $\Lambda$ as in
\eqref{eq:set-s}. Then a generating vector $\bsz^*=(z_1,\ldots,z_d) \in
(\ZZ^*_n)^d$ can be constructed component-by-component such that for all
$s = 1, \ldots, d$ and $\bsz = (z_1, \ldots, z_s)$ we have
\begin{align}\label{eq:good-z-c}
  \bssigma(\bsk')\cdot\bsz\not\equiv_n \bsk\cdot\bsz
  \qquad
  \text{for all } \bsk,\bsk' \in \Lambda_s \mbox{ and } \bssigma\in \calS_{\bsk'}
  \text{ with } \bsk \ne \bsk'
  .
  \end{align}
\end{theorem}

\begin{proof}
The proof is by induction on $s$. For $s=1$ and $k_1\ne k_1'$, the
condition $\sigma_1(k_1') z_1\not\equiv_n k_1 z_1$ holds for all $z_1 \in
\ZZ^*_n$ if $\sigma_1(k_1')-k_1\not\equiv_n 0$, and fails for all $z_1$ if
$\sigma_1(k_1')-k_1\equiv_n 0$. To avoid the latter scenario we assume
that $n>2\max_{k_1\in\Lambda_1} |k_1|$.

Suppose we already obtained the generating vector $\bsz \in
(\ZZ^*_n)^{s-1}$ satisfying \eqref{eq:good-z-c} for $\Lambda_{s-1}$. For
each distinct pair $(\bsk, k_s),(\bsk', k_s') \in \Lambda_s$ and each
$(\bssigma,\sigma_s)\in \calS_{(\bsk',k_s')}$, we will eliminate any `bad'
$z_s\in\bbZ_n^*$ that satisfies
\begin{align}\label{eq:bad-z-c}
     (\bssigma(\bsk'), \sigma_s(k_s')) \cdot (\bsz, z_s)
     \equiv_n
     (\bsk, k_s) \cdot (\bsz, z_s)
     \quad\Leftrightarrow\quad
     (\sigma_s(k_s') - k_s) z_s
     \equiv_n - (\bssigma(\bsk')\cdot\bsz - \bsk\cdot\bsz).
\end{align}
{}From the definition \eqref{eq:set-s} we have $\bsk,\bsk' \in
\Lambda_{s-1}$. We have the following scenarios:
\begin{enumerate}
\item If $\bsk\ne\bsk'$ then the induction hypotheses
    \eqref{eq:good-z-c} for $\Lambda_{s-1}$ guarantees that
    $\bssigma(\bsk')\cdot\bsz \not\equiv_n \bsk\cdot\bsz$ for all
    $\bssigma$. Thus \eqref{eq:bad-z-c} has no solution for~$z_s$ if
    $\sigma_s(k_s')-k_s\equiv_n 0$, and \eqref{eq:bad-z-c} has a
    unique solution for $z_s$ if $\sigma_s(k_s')- k_s\not\equiv_n 0$.

\item If $\bsk=\bsk'$ (thus $k_s\ne k_s'$) and $\bssigma$ satisfies
    $\bssigma(\bsk')\cdot\bsz \not\equiv_n \bsk\cdot\bsz$ then, as in
    the previous scenario, \eqref{eq:bad-z-c} has no solution
    for~$z_s$ if $\sigma_s(k_s')-k_s\equiv_n 0$, and
    \eqref{eq:bad-z-c} has a unique solution for $z_s$ if
    $\sigma_s(k_s')- k_s\not\equiv_n 0$.

\item If $\bsk=\bsk'$ (thus $k_s\ne k_s'$) and $\bssigma$ satisfies
    $\bssigma(\bsk')\cdot\bsz \equiv_n \bsk\cdot\bsz$, then
    \eqref{eq:bad-z-c} has no solution for~$z_s$ if
    $\sigma_s(k_s')-k_s\not\equiv_n 0$, and \eqref{eq:bad-z-c} holds
    for all $z_s$ if $\sigma_s(k_s')- k_s\equiv_n 0$. To avoid the
    latter scenario we assume that $n > |\sigma_s(k_s')- k_s|$. Note
    that it is not possible to have $\sigma_s(k_s') = k_s$ when
    $k_s\ne k_s'$ since both $k_s$ and $k_s'$ are nonnegative
    integers.
\end{enumerate}
Thus, provided that $n > 2\max_{(\bsk,k_s)\in\Lambda_s} |k_s|$, there is
at most one bad $z_s$ to be eliminated for each distinct pair $(\bsk,
k_s),(\bsk', k_s') \in \Lambda_s$ and each $(\bssigma,\sigma_s)\in
\calS_{(\bsk',k_s')}$, so the total number of bad $z_s$ we eliminate is at
most $\size{\Lambda_s}\,(\size{\calM(\Lambda_s)} - 1)$.

Hence, provided additionally that $\size{\ZZ^*_n} = n-1 >
\size{\Lambda_s}\,(\size{\calM(\Lambda_s)} - 1)$, there is always a `good'
$z_s$ remaining such that $(\bsz,z_s)$ will satisfy~\eqref{eq:good-z-c}
for $\Lambda_s$. By induction, to ensure that a good
$\bsz^*\in(\bbZ_n^*)^d$ exists, it suffices to assume that $n$ satisfies
\eqref{eq:cond-n-c}. This completes the proof.
\end{proof}

\section{Approximation} \label{sec:app}

We now discuss function approximation for all three settings under a
unified framework. A major difference of this section compared to the
previous sections is that the function $f$ under consideration is no
longer supported only on a finite index set. We cannot achieve exact
function reconstruction and therefore an error analysis is needed.

\subsection{Function approximation under a unified framework}

We have an orthonormal basis $\{\alpha_\bsk\}$ for $L_\mu^2(\Omega)$,
where $\mu(\Omega) = 1$, and consider functions with absolutely converging
series expansions
\[
  f \,=\,
  \sum_\bsk \widehat{f}_\bsk \, \alpha_\bsk,
\]
where the sum is over $\bbZ^d$ (for Fourier space) or $\bbN_0^d$ (for
cosine and Chebyshev spaces). Now consider a subset $\Lambda$ of the
indices and represent the exact $L_\mu^2$ projection of $f$, i.e., the
best $L_\mu^2$ approximation on $\Lambda$, by
\[
  f_\Lambda
  \,=\,
  \sum_{\bsk \in \Lambda} \widehat{f}_\bsk \, \alpha_\bsk.
\]
We cannot calculate these coefficients $\widehat{f}_\bsk$ exactly and will
have to approximate them, leading to
\[
  f_\Lambda^{\square}
  \,=\,
  \sum_{\bsk \in \Lambda} \widehat{f}_\bsk^{\square} \, \alpha_\bsk,
\]
with $\square\in \{a,b,c\}$ denoting the approximation by plan A
(including the Fourier case), plan B, or plan C.

We write
\begin{align*}
  \varphi
  &=
  \begin{cases}
  {\rm id} & \hspace{1.78cm} \text{for Fourier space}, \\
  \tent & \hspace{1.78cm} \text{for cosine space plan A, B, C}, \\
  \cos(\pi\,\tent(\bigcdot)) & \hspace{1.78cm} \text{for Chebyshev space plan A, B, C},
  \end{cases}
  \\
  \alpha_\bsk \circ \varphi = \rmu_\bsk
  &=
  \begin{cases}
    \exp(2\pi\rmi \, \bsh \cdot \bsx) & \text{for Fourier space}, \\
    \sqrt{2}^{|\bsk|_0} \prod_{j=1}^d \cos(2\pi \, k_j x_j) & \text{for cosine/Chebyshev space plan A, B, C},
  \end{cases}
  \\
  \rmv_\bsk
  &=
  \begin{cases}
    \rmu_\bsk & \hspace{0.9cm} \text{for Fourier space, cosine/Chebyshev space plan A}, \\
    \sqrt{2}^{|\bsk|_0} \cos(2\pi \, \bsk\cdot\bsx) & \hspace{0.9cm} \text{for cosine/Chebyshev space plan B, C}.
  \end{cases}
\end{align*}
Then we have $\langle {\rm u}_\bsk,\rmu_{\bsk'}\rangle =
\delta_{\bsk,\bsk'}$, $\langle \rmu_\bsk,\rmv_{\bsk'}\rangle =
\delta_{\bsk,\bsk'}$, and $\langle \rmv_\bsk,\rmv_{\bsk'}\rangle =
d_\bsk\, \delta_{\bsk,\bsk'}$, with $d_\bsk = 2^{|\bsk|_0-1}$ for
$\bsk\ne\bszero$ in the case of cosine or Chebyshev space plan B or C and
$d_\bsk = 1$ otherwise.

We demand from our lattice rule that
\begin{align}
  Q_n(\rmu_\bsk \, \overline{\rmv_{\bsk'}}) \label{eq:Qalphabeta}
  &=
  c_\bsk \, \delta_{\bsk,\bsk'}
  \qquad
  \forall \bsk, \bsk' \in \Lambda,
  \\
  Q_n(\rmv_\bsk \, \overline{\rmv_{\bsk'}}) \label{eq:Qbeta}
  &=
  d_\bsk \, \delta_{\bsk,\bsk'}
  \qquad
  \forall \bsk, \bsk' \in \Lambda
  ,
\end{align}
where $c_\bsk$ in the case of cosine or Chebyshev space plan C (see
\eqref{eq:plan-c-ck} or \eqref{eq:che-plan-c-ck}) can be a positive
integer up to the number of unique sign changes of $\bsk$, i.e., $1\le
c_\bsk\le 2^{|\bsk|_0}$, and $c_\bsk = 1$ otherwise. Note that we do not
necessarily have $Q_n(\rmu_\bsk \, \overline{\rmu_{\bsk'}}) =
\delta_{\bsk,\bsk'}$ (except for when $\rmv_\bsk = \rmu_\bsk$).

With the above unifying notation, and with $w_i = 1/n$ and $\bst_i$ our
lattice points, we can write our approximate coefficient as
\begin{align}
  \label{eq:approx}
  \widehat{f}_\bsk^{\square}
  \,=\, Q_n(f\circ\varphi\, \overline{\rmv_\bsk})\,c_\bsk^{-1}
  \,=\,
  \sum_{i=0}^{n-1} w_i\,f(\varphi(\bst_i)) \, \overline{\rmv_\bsk(\bst_i)}\,c_\bsk^{-1}
  ,\qquad \bsk\in\Lambda\,.
\end{align}
In comparison, the exact coefficient is given by $\widehat{f}_\bsk =
  \langle f , \alpha_\bsk \rangle_\mu
  =
  \langle f\circ\varphi, \rmu_\bsk \rangle
  =
  \langle f\circ\varphi, \rmv_\bsk \rangle$.

\subsection{Connection to discrete least squares}

With a prescribed ordering of the elements in $\Lambda$, the approximate
coefficients \eqref{eq:approx} for $\bsk\in\Lambda$ can be written in
matrix-vector notation as
\[ 
  \widehat{\bsf}^\square
  \,=\,
  C^{-1} \, V^* \, W \, \bsf_\varphi
  ,
\] 
with column vectors $\bsf_\varphi = [f(\varphi(\bst_i))]_i$,
$\widehat{\bsf}^\square = [\widehat{f}_\bsk^{\square}]_{\bsk}$, and
matrices $V = [\rmv_\bsk(\bst_i)]_{i,\bsk}$, $C = {\rm diag}(c_\bsk)$, and
$W = {\rm diag}(w_i)$. The conditions \eqref{eq:Qalphabeta} and
\eqref{eq:Qbeta} can be expressed as
\[
  V^* \, W \, U = C
  \qquad\mbox{and}\qquad
  V^* \, W \, V = D,
\]
with matrices $U = [\rmu_\bsk(\bst_i)]_{i,\bsk}$ and $D = {\rm
diag}(d_\bsk)$.

For plan A (or for the Fourier case) we have $U=V$ and $C = D = I$, so
\begin{align*}
    \widehat{\bsf}^a
  &=
  U^* \, W \, \bsf_\varphi
 ,
\end{align*}
which is precisely the solution to the normal equations
\[
  (W^{1/2} \, U)^* \, W^{1/2} \, U \, \widehat{\bsf}^a
  =
  (W^{1/2} \, U)^* \, W^{1/2} \, \bsf_\varphi
  \Leftrightarrow
  (U^* \, W \, U) \, \widehat{\bsf}^a
  =
  U^* \, W \, \bsf_\varphi
  \Leftrightarrow
  \widehat{\bsf}^a
  =
  U^* \, W \, \bsf_\varphi
  ,
\]
which in turn solves the discrete least-squares problem
\[
  \min_{\widehat{\bsf}^a} \| W^{1/2} \, U \, \widehat{\bsf}^a  - W^{1/2} \, \bsf_\varphi \|^2_2.
\]

Plan B and plan C do \emph{not} have the least-squares interpretation.

\subsection{Stability to perturbation}

Suppose that there is perturbation error in the function evaluations in
\eqref{eq:approx} so that instead of $f(\varphi(\bst_i))$ we have
\[
   f_{\rm pert}(\varphi(\bst_i)) \,=\, f(\varphi(\bst_i)) + \eps_i,
   \qquad i=0,\ldots,n-1.
\]
We denote the corresponding perturbed approximate coefficients by
$\widehat{f}_{\bsk,{\rm pert}}^{\square}$ and the corresponding
approximate function over $\Lambda$ by $f_{\Lambda,{\rm pert}}^{\square}$
for $\square\in \{a,b,c\}$. Using \eqref{eq:approx}, we can write
\begin{align*}
  \|f_{\Lambda,{\rm pert}}^\square - f_\Lambda^\square\|^2_{L^2_\mu}
  &\,=\,
  \sum_{\bsk \in \Lambda} |\widehat{f}_{\bsk,{\rm pert}}^\square - \widehat{f}_\bsk^\square|^2 \\
  &\,=\,
  \sum_{\bsk \in \Lambda} \left| \sum_{i=0}^{n-1} w_i\, (f_{\rm pert} - f)(\varphi(\bst_i)) \,
  \overline{\rmv_\bsk(\bst_i)}\, c_\bsk^{-1} \right|^2
  \,=\, \|C^{-1} V^* W \bseps\|_2^2,
\end{align*}
with column vector $\bseps = [\eps_i]_i$. We have
\begin{align*}
  \|C^{-1} V^* W \bseps\|_2^2
  \,=\,
  \|C^{-1} \, D^{1/2} \, ( D^{-1/2} \, V^* \, W^{1/2} ) \, W^{1/2} \, \bseps\|_2^2
  ,
\end{align*}
where
\[
  ( D^{-1/2} \, V^* \, W^{1/2} ) \, ( W^{1/2} \, V \, D^{-1/2} )
  \,=\, D^{-1/2} \, V^* \, W \, V \, D^{-1/2} = I
  ,
\]
so that $\|D^{-1/2} \, V^* \, W^{1/2}\|_2 = 1$. Thus
\begin{align*}
    \|f_{\Lambda,{\rm pert}}^\square - f_\Lambda^\square\|^2_{L^2_\mu}
  \,=\, \|C^{-1} V^* W \bseps\|_2^2
  \,\le\,
  \|C^{-1} \, D^{1/2}\|_2^2 \, \|W^{1/2}\,\bseps\|_2^2
  &\,=\, \bigg(\underbrace{\max_{\bsk\in\Lambda} \frac{d_\bsk}{c_\bsk^2}}_{=:\,\rho_\Lambda^{\square}}\bigg)
  \bigg(
  \frac{1}{n} \sum_{i=0}^{n-1} |\eps_i|^2
  \bigg).
\end{align*}
Here $\rho_\Lambda^{\square}$ is the \emph{stability constant}, and we
have
\begin{equation} \label{eq:stab}
  \rho_\Lambda^{a} \,=\, 1,
  \quad
 \rho_\Lambda^{b} \,=\, \max\left(1_{\bszero\in\Lambda},\max_{\bsk\in\Lambda\setminus\{\bszero\}} 2^{|\bsk|_0-1}\right),
  \quad
 \rho_\Lambda^{c} \,=\, \max\left(1_{\bszero\in\Lambda},\max_{\bsk\in\Lambda\setminus\{\bszero\}} \frac{2^{|\bsk|_0-1}}{c_\bsk^2}\right),
\end{equation}
where $1_{\bszero\in\Lambda}$ is 1 if $\bszero\in\Lambda$ and is $0$
otherwise.

For plan~A the approximation is perfectly stable.

For plan~B and plan~C we have the general upper bound $\rho_\Lambda^{c}\le
\rho_\Lambda^{b}\le 2^{d-1}$ which might be too pessimistic. If we have a
weighted index set with decaying weights (see Example~\ref{example})
$\rho_\Lambda^{b}$ could be much smaller. Alternatively, if $\Lambda$ is
downward closed then by Lemma~\ref{lem:downward} we have
$\rho_\Lambda^{c}\le\rho_\Lambda^{b}\le \size{\Lambda}$. In any case, the
stability constant for plan~B is likely to be much bigger than $1$ even if
it is independent of $d$.

For plan~C the values of $c_\bsk$ depend on the lattice rule and can
potentially be as large as $2^{|\bsk|_0}$, giving hope that one may
attempt to minimize the stability constant $\rho_\Lambda^{c}$ as part of
the CBC construction of the lattice generating vector. Unfortunately,
numerical experiments show that not much improvement can be obtained
because ``self-aliasing'' does not happen often enough.

\subsection{Error analysis}

We have for our three plans A (including the Fourier space), B, and C,
annotated by $\square\in\{a,b,c\}$,
\begin{align*}
  \|f - f_\Lambda^\square\|^2_{L^2_\mu}
  =
  \underbrace{\|f - f_\Lambda\|^2_{L^2_\mu}}_{\text{truncation error}}
  +
  \underbrace{\|f_\Lambda - f_\Lambda^\square\|^2_{L^2_\mu}}_{\text{approximation error}}
  .
\end{align*}

The first part is the truncation error for the finite index set $\Lambda$,
and thus represents the best $L_\mu^2$ approximation error for the choice
of $\Lambda$. In the \emph{Information Based Complexity} (IBC) error
analysis, this would be a complexity result using \emph{arbitrary linear
information}: if we know more about the smoothness class of our functions,
then this bound is known in terms of a set $\Lambda$ which is constructed
according to the decay of the singular values of the approximation
operator and this error is exactly the next singular value, see, e.g.,
\cite{WW99}. The second part is how well we approximate this best possible
approximation by our numerical algorithm which only uses function values;
in IBC this is known as \emph{standard information}, see, e.g.,
\cite{WW01}.

We proceed to analyze the second error $\|f_\Lambda -
f_\Lambda^\square\|^2_{L^2_\mu}$. Since $f_\Lambda$ is supported only on
$\Lambda$, our reconstruction lattice can exactly compute its coefficients
on $\Lambda$. Thus for $\bsk\in\Lambda$ we have
\begin{align}
  \label{eq:exact}
  \widehat{f}_\bsk
  \,=\, \widehat{(f_\Lambda)}_\bsk
  \,=\, Q_n(f_\Lambda\circ\varphi\, \overline{\rmv_\bsk})\,c_\bsk^{-1}
  \,=\,
  \sum_{i=0}^{n-1} w_i\,f_\Lambda(\varphi(\bst_i)) \, \overline{\rmv_\bsk(\bst_i)}\,c_\bsk^{-1}
  .
\end{align}
Using \eqref{eq:exact} and \eqref{eq:approx} and following the same
argument as for the stability analysis, we obtain
\begin{align*}
  \|f_\Lambda - f_\Lambda^\square\|^2_{L^2_\mu}
  \,=\,
  \sum_{\bsk \in \Lambda} |\widehat{f}_\bsk - \widehat{f}_\bsk^\square|^2
  &\,=\,
  \sum_{\bsk \in \Lambda} \left| \sum_{i=0}^{n-1} w_i\, (f_\Lambda - f)(\varphi(\bst_i)) \, \overline{v_\bsk(\bst_i)}\, c_\bsk^{-1} \right|^2
  \,=\, \|C^{-1} V^* W \bsg\|_2^2,
\end{align*}
with column vector $\bsg = [(f_\Lambda - f)(\varphi(\bst_i))]_i$, and we
arrive at
\begin{align} \label{eq:d-norm}
  \|f_\Lambda - f_\Lambda^\square\|^2_{L^2_\mu}
  &\,\le\, \rho_\Lambda^{\square}\, \|f - f_\Lambda\|_n^2,
  \qquad\mbox{where}\quad
  \|h\|_n^2 \,:=\, \frac{1}{n} \sum_{i=0}^{n-1} |h(\varphi(\bst_i))|^2 .
\end{align}

We summarize our combined result for function approximation in the
following theorem.

\begin{theorem} \label{thm:final}
Consider the problem of approximating a function $f\in\calF$ from the
Fourier, cosine, or Chebyshev space by $f^{\square}_\Lambda$, $\square\in
\{a,b,c\}$, using a finite index set $\Lambda$ and an $n$-point rank-$1$
lattice under plans A, B, or C as described in this paper. For
sufficiently large $n$ we have
\begin{align} \label{eq:err-good}
  \|f - f_\Lambda^\square\|^2_{L^2_\mu}
  &\,\le\,
  \|f - f_\Lambda\|^2_{L^2_\mu} + \rho_\Lambda^{\square}\, \|f - f_\Lambda\|_n^2,
\end{align}
with stability constant $\rho_\Lambda^{\square}$ given in \eqref{eq:stab}
and discrete seminorm $\|\cdot\|_n$ defined in \eqref{eq:d-norm}. We have
$\rho_\Lambda^{a} = 1$, and if $\Lambda$ is downward closed then
$\rho_\Lambda^{c} \le \rho_\Lambda^{b} \le \min(2^{d-1},\size\Lambda)$. A
loose upper bound 
is
\begin{align} \label{eq:err-bad}
  \|f - f_\Lambda^\square\|_{L^2_\mu}
  \,\le\,
  \sqrt{1 + \rho_\Lambda^{\square}}\; \|f - f_\Lambda\|_{L^\infty}.
\end{align}
The requirement on $n$ to achieve \eqref{eq:err-good} is proportional to
$\sizep{\Lambda\ominus\Lambda}$ for the Fourier case, while for the
cosine/Chebyshev case it is $\sizep{\calM(\Lambda)\oplus\calM(\Lambda)}$
with plan A, $\sizep{\Lambda\oplus\calM(\Lambda)}$ with plan B, and
$\size{\Lambda}\,\size{\calM(\Lambda)}$ with plan~C.
\end{theorem}

\subsection{Comparison with previous results from the literature and conclusions}

Approximation by discrete least-squares has been analysed for different
measures $\mu$ and bases $\{ \alpha_{\bsk}\}$ in several works. We mention
for instance \cite{CCMNT15,CDL13,MNST14} when using evaluations at random
points, and in \cite{MN15,NXZ14} when random points are replaced by
deterministic point sets. A common denominator in all the aforementioned
analyses is the equivalence of the norm $\| \cdot \|_{L^2_\mu}$ and a
suitably defined discrete seminorm $\|\cdot\|_n$, on the
finite-dimensional space $\calF_\Lambda$. More precisely, there exists
$\delta\in [0,1)$ such that, under appropriate conditions on $n$,
$\#\Lambda$ and $\delta$, it holds that
\begin{equation}
(1-\delta)
\,
\|f\|^2_{L^2_\mu}
\leq
\|f\|_n^2
\leq
(1+\delta)
\,
\|f\|^2_{L^2_\mu},
\qquad\mbox{for all}\quad f\in \calF_\Lambda.
\label{eq:norm_equivalence}
\end{equation}
Under the same conditions between $n$, $\#\Lambda$ and $\delta$ that
ensure \eqref{eq:norm_equivalence}, the discrete least-squares
approximation $\Pi_n f$ of any $f\in\calF$ satisfies
\begin{align} \label{eq:lsq}
  \|f - \Pi_n f\|_{L^2_\mu}
  \,\le\, \sqrt{1 + \frac{1}{1-\delta}}\; \inf_{v\in\calF_\Lambda} \|f - v\|_{L^\infty},
\end{align}
see \cite[Proposition 1]{MNST14} for a proof.

Our results for the Fourier space and for plan A of the cosine and
Chebyshev spaces achieve exactly $\delta = 0$; see also \eqref{eq:err-bad}
with $\rho_\Lambda^a = 1$. For the Fourier case we obtain essentially a
scaling of $n\ge (\size{\Lambda})^2$. For the cosine and Chebyshev spaces
we obtain essentially $n\ge 4^d\,(\size{\Lambda})^2$ in general, and $n\ge
\min(4^d\,(\size{\Lambda})^2,(\size{\Lambda})^{2\ln 3/\ln 2})$ for
downward closed index sets. However, if the mirrored index set is a
weighted hyperbolic cross with sufficiently fast decaying weights (see
Example~\ref{example}), then we obtain essentially $n\ge
c_\tau\,(\size{\Lambda})^{2\tau}$ for $\tau>1$ arbitrarily close to $1$.

With the Chebyshev space and for any downward closed set $\Lambda$, these
results improve on \cite{NXZ14} where it is proven that
\eqref{eq:norm_equivalence} holds true if $n \geq 2^{2d+1} d^2
(\#\Lambda)^2$.

Moreover, we mention that in the case of uniform measure $\mu$ and
expansion on the Legendre basis, the results in \cite{MN15} show a scaling
of $n/ (\ln n)^d$ as $(\#\Lambda)^4$ for general downward closed sets, and
a scaling of $n/ (\ln n)^d$ as $(\#\Lambda)^2$ when $\calF_\Lambda$ is an
anisotropic tensor product space.

Our results for plans B and C in the cosine and Chebyshev spaces do not
have the discrete least-squares interpretation. All three plans give exact
function reconstruction in $\calF_\Lambda$, but for a general $f\in\calF$
not finitely supported on $\Lambda$, there is a trade-off between the
approximation error and the requirement on $n$ (e.g., plan $A$ requires
$n$ to be larger but also has the smallest constant~$\rho_\Lambda^a$).
Therefore it is not easy to directly compare them without further
analysis.

To proceed further from the very general result in
Theorem~\ref{thm:final}, one would need to make further assumptions on,
for example, the smoothness properties of the function space, and the
knowledge of a corresponding index set that has been chosen to take
advantage of such properties. Starting from the loose upper bound
\eqref{eq:err-bad}, if we know that the best $L^\infty$ approximation
error satisfies $\|f - f_\Lambda\|_{L^\infty} \le c_q\,
(\size\Lambda)^{-q}$ for some $q>0$, see, e.g., \cite{KWW08}, then we
arrive at $\|f - f_\Lambda^\square\|_{L^2_\mu}
  \,\le\,
  \sqrt{1 + \rho_\Lambda^{\square}}\; c_q\,(\size\Lambda)^{-q}$.
For the Fourier space we have $\rho_\Lambda^a=1$ and $n$ needs to be
proportional to $\sizep{\Lambda\ominus\Lambda}\le (\size{\Lambda})^2$,
leading to $\|f - f_\Lambda^a\|_{L^2_\mu} = \calO(n^{-q/2})$, where the
implied constant is independent of $d$ if $c_q$ is independent of $d$. For
the cosine or Chebyshev space, the result is more complicated because it
depends on the size of the mirrored index set. If we have a weighted
hyperbolic cross with sufficiently decaying weights (see
Example~\ref{example}) then the mirrored set itself is not of concern.
However, for plan B or C we need to further take into account the value of
$\rho_\Lambda^{b}$ or $\rho_\Lambda^{c}$. In general $\rho_\Lambda^{b}$
and $\rho_\Lambda^{c}$ can be much worse than $\rho_\Lambda^a=1$, but
depending on the actual index set they might also be manageable.

Finally we stress that the $L^2$~approximation result based on \emph{the
estimate \eqref{eq:err-bad} is not sharp, and neither is \eqref{eq:lsq}},
because the best $L^2$~approximation error $\|f - f_\Lambda\|_{L^2_\mu}$,
i.e., the first term on the right-hand side of \eqref{eq:err-good}, has
been estimated by the best $L^\infty$~approximation error $\|f -
f_\Lambda\|_{L^\infty}$, which is generally half an order worse in the
convergence rate (e.g., rate $p$ for $L^2$ versus rate $p-1/2=q$ for
$L^\infty$), see, e.g., \cite{KWW09a}. Moreover, a direct analysis on the
discrete norm $\|f - f_\Lambda\|_n$, i.e., the second term on the
right-hand side of \eqref{eq:err-good}, based on properties of the lattice
points, has a chance to improve upon the best $L^\infty$~approximation
error too. Indeed, function approximation based on rank-$1$ lattices has
been analyzed in \cite{CKNS16,KSW06,KSW08,KWW09c} where the lattice
generating vectors were constructed to minimize the approximation error
directly, without the reconstruction property. It is known that if $p$ is
the rate of convergence for the best $L^2$~approximation error (rather
than $L^\infty$) then lattice generating vectors can be constructed to
achieve $\|f - f_\Lambda^a\|_{L^2_\mu} = \calO(n^{-p/2})$. There are also
other approximation results using rank-$1$ or multiple rank-$1$
lattices, see, e.g., \cite{BKUV17,KDV15,KKP12,KV19,LH03,ZLH06,ZKH09}.

Rank-$1$ lattices are very attractive due to their simplicity and
stability, and the availability of fast computation methods compared to
other approximation algorithms.


\end{document}